\documentclass[12pt,oneside, reqno]{amsart}
\usepackage[utf8]{inputenc}
\usepackage[english]{babel}
\usepackage{multicol}
\usepackage{setspace}

\usepackage{setspace}
\usepackage{hyperref}
\usepackage{amssymb}
\usepackage{amsmath,mathdots}
\usepackage{amsthm}
\usepackage{stmaryrd}
\usepackage{graphicx}
\usepackage{mathrsfs}
\usepackage[noabbrev]{cleveref}
\usepackage{mathtools,leftindex}
\usepackage{tikz-cd}

\usepackage{imakeidx}
\makeindex

\usepackage{xcolor}
\hypersetup{
    colorlinks=true,
    linkcolor=blue,
    filecolor=magenta,      
    urlcolor=cyan,
    citecolor=blue,
}
\usepackage[textwidth=3cm, textsize=small, colorinlistoftodos]{todonotes}

\usepackage{enumerate}

\input xy
\xyoption{all}
\usepackage[color,matrix,arrow]{xy}
\usepackage{tikz}
\usetikzlibrary{bending}
\usetikzlibrary{arrows.meta}
\usetikzlibrary{matrix,arrows,decorations.markings, calc, positioning}
\usetikzlibrary {shapes.geometric}
\usetikzlibrary {intersections}

\usepackage{geometry}
\makeatletter
\providecommand{\leftsquigarrow}{%
  \mathrel{\mathpalette\reflect@squig\relax}%
}
\newcommand{\reflect@squig}[2]{%
  \reflectbox{$\m@th#1\rightsquigarrow$}%
}
\makeatother

\let\epsilon=\relax

\newcommand{\C}{C_{p^rq^s}}

\newcommand{\Cat}{\text{Cat}}
\newcommand{\epsilon}{\varepsilon}

\newcommand{\Hull}{\text{Hull}}

\let\ij=\relax
\newcommand{\ij}[1]{{\normalfont \scalebox{2}[1]{[}}(#1){\normalfont \scalebox{2}[1]{]}}}

\let\0=\relax
\newcommand{\0}{\ij{0,0}}

\usepackage{multicol}
\usepackage{pdfpages}
\usepackage{cancel}

\theoremstyle{plain}
\newtheorem{theorem}{Theorem}[section]
\newtheorem{lemma}[theorem]{Lemma}
\newtheorem{proposition}[theorem]{Proposition}
\newtheorem{corollary}[theorem]{Corollary}
\newtheorem*{theorem*}{Theorem}

\theoremstyle{definition}
\newtheorem{definition}[theorem]{Definition}

\newtheorem{example}[theorem]{Example}
\newtheorem{remark}[theorem]{Remark}

\newtheorem{notation}[theorem]{Notation}

\makeatletter
\let\c@equation\c@theorem
\numberwithin{equation}{section}
\makeatother

\crefname{lemma}{Lemma}{Lemmas}
\crefname{theorem}{Theorem}{Theorems}
\crefname{definition}{Definition}{Definitions}
\crefname{proposition}{Proposition}{Propositions}
\crefname{remark}{Remark}{Remarks}
\crefname{corollary}{Corollary}{Corollaries}
\crefname{example}{Example}{Examples}

\title{Uniquely Compatible Transfer Systems for Cyclic Groups of Order $p^rq^s$}

\author{Kristen Mazur}
\address[Mazur]{Department of Mathematics and Statistics, Elon University, Elon, NC 27244, USA}
\email{kmazur@elon.edu}

\author{Ang\'{e}lica M.~Osorno}
\address[Osorno]{Department of Mathematics and Statistics, Reed College, Portland, OR 97202, USA}
\email{aosorno@reed.edu}

\author{Constanze Roitzheim}
\address[Roitzheim]{School of Mathematics, Statistics and Actuarial Science, University of Kent, Kent CT2 7NF, UK}
\email{c.roitzheim@kent.ac.uk}

\author{Rekha Santhanam}
\address[Santhanam]{Department of Mathematics, IIT Bombay, Powai, Mumbai 400076, India}
\email{reksan@iitb.ac.in}

\author{Danika Van Niel}
\address[Van Niel]{Department of Mathematics, Michigan State University, East Lansing, MI 48824, USA}
\email{vannield@msu.edu}

\author{Valentina Zapata Castro}
\address[Zapata Castro]{Department of Mathematics, University of Virginia, Charlottesville, VA 22904, USA}
\email{vz6an@virginia.edu}

\begin{document}

\begin{abstract}
 Bi-incomplete Tambara functors over a  group $G$ can be understood in terms of  compatible pairs of $G$-transfer systems.
 In the case of $G=C_{p^n}$,  Hill, Meng and Li gave a necessary and sufficient condition for compatibility and computed the exact number of compatible pairs. 
 In this article, we study compatible pairs of $G$-transfer systems for the case $G=C_{p^rq^s}$ and identify conditions when such transfer systems are uniquely compatible in the sense that they only form trivially compatible pairs. This gives us new insight into collections of norm maps that are relevant in equivariant homotopy theory.  
\end{abstract}

\maketitle


\section{Introduction}

For a finite group $G$, the equivariant stable category is highly complex. For instance, the stable category is not  uniquely determined by the group. The choice depends on which  $G$-orbits  are required to be dualizable, and this choice affects which transfer maps, i.e., which ``wrong-way maps'' between fixed-points for nested subgroups, become  part of the data of equivariant spectra.

Even in the \emph{genuine} equivariant stable category, where all the orbits are dualizable and spectra support all possible transfer maps, there is variation in the notion of $E_\infty$ ring spectra. As noted by Blumberg and Hill \cite{BHOperads}, highly structured commutative ring $G$-spectra can support a range of norm maps, which are neatly encoded by $N_\infty$ $G$-operads.

Blumberg and Hill show in \cite{BHOperads} that an $N_\infty$ $G$-operad is characterized up to homotopy by the collection of wrong-way maps
(transfers if additive, norms if multiplicative) it encodes. Since transfers and norms are indexed by nested pairs  of subgroups of $G$, keeping track of the transfers and norms gives rise to the notion of a \emph{$G$-transfer system}, which is a subposet of the poset of subgroups of $G$ under inclusion satisfying certain closure conditions, see \cref{def:tranfsys}. Furthermore, the combined work of Balchin-Barnes-Roitzheim \cite{NinftyOperads}, Blumberg-Hill \cite{BHOperads}, Bonventre-Pereira \cite{BPGenuineEquivaiantOperads}, Guti\'errez-White \cite{GWViaOperads}, and Rubin \cite{RubinCombNinftyOperads,RubinDetectingOperads} shows that the homotopy category of $N_\infty$ $G$-operads is equivalent to the poset of $G$-transfer systems (given by inclusion). Thus, $G$-transfer systems can be thought of as combinatorial gadgets that control homotopy commutative operations in the equivariant setting. 

Given an $N_\infty$ $G$-operad $\mathcal{O}$ and an $\mathcal{O}$-algebra $X$ in $G$-spaces, Blumberg and Hill \cite{BHOperads} show that $\underline{\pi}_0(X)$ has the structure of an incomplete Mackey functor, with transfer maps generated precisely by those encoded in the $G$-transfer system associated to $\mathcal{O}$. Similarly, if $R$ is an $\mathcal{O}$-algebra in the category of genuine $G$-spectra, they show that $\underline{\pi}_0(R)$ is an incomplete Tambara functor, with norms generated by those in the $G$-transfer system \cite{BHIncomplete}. It is then natural to study what happens when both the additive and the multiplicative structures are incomplete, i.e., study $N_\infty$ algebras in an incomplete category of $G$-spectra. 

Blumberg and Hill \cite{Biincomplete} study the algebraic analogue: bi-incomplete Tambara functors, which have incomplete collections of transfers and norms. Since in the classical theory of Tambara functors the norm of a sum depends on certain transfers (see \cite[Theorems 2.4, 2.5]{HillMazur}), it is not surprising that the presence of a given norm implies the existence of certain transfers. Combined work of Blumberg-Hill \cite{Biincomplete} and Chan \cite{ChanTambara}, gives precise combinatorial conditions on a pair $(T_m,T_a)$ of $G$-transfer systems that imply that $T_m$ gives the multiplicative norms and $T_a$ gives the additive transfers of a bi-incomplete Tambara functor. Such pairs of $G$-transfer systems are called \emph{compatible}.

 Other prior work on compatible pairs consists of enumerations; Hill, Meng, and Li \cite{Hill-Meng-Nan} for $G=C_{p^r}$ and Ormsby \cite{Ormsby}  for $G=C_p \times C_p$. Both results are based on previous enumerations of the corresponding $G$-transfer systems \cite{NinftyOperads,Baoetal}. In this paper, we study compatible pairs of $G$-transfer systems when $G$ is in the family of cyclic groups of the form $\C$ with $p$ and $q$ distinct primes. The enumeration of transfer systems for $C_{p^r q^s}$ is notoriously difficult (see \cite[Section 4]{NinftyOperads}), so instead of enumerating compatible pairs of $\C$-transfer systems, this paper focuses on determining when a $\C$-transfer system is the multiplicative part of a compatible pair in  a certain unique way.

Indeed, every  $G$-transfer system is multiplicatively compatible with two (not necessarily distinct)  $G$-transfer systems. First, all are compatible with the \emph{complete} $G$-transfer system $T_c$, which contains all subgroup relations.   Second, given a $G$-transfer system $T$, there is a minimal (in terms of inclusion) $G$-transfer system $\Hull(T)$ such that $(T, \Hull(T))$ is compatible. We thusly think of $(T,T_c)$ and $(T,\Hull(T))$ as being the two \textit{trivially compatible} pairs for any given $T$. Our first main result identifies those transfer systems for which these two coincide. The conditions for this theorem depend on the number and shape of the \emph{connected components} of $T$, which are the path components of the graph representing $T$.

\begin{theorem*}[\Cref{cor:connected} and \cref{thm:OneCompLSP}]
  Let $G$ be any finite group and $T$ be a $G$-transfer system. Then $\Hull(T)=T_{c}$ if and only if $T$ is connected. In this case, $T$ is the multiplicative part of exactly one compatible pair.
  \end{theorem*}

We further identify necessary and sufficient conditions for when any $\C$-transfer system $T$ is only the multiplicative part of the two trivially compatible pairs. In other words, we identify conditions   for when $T$ is only multiplicatively compatible with $T_c$ and $\Hull(T)$. We call such a $T$ \emph{lesser simply paired.} The conditions continue to depend on the number and shape of the connected components of $T$.  By the above theorem,  if a $\C$-transfer system $T$ has a single connected component then it is lesser simply paired. The following theorem gives  conditions under which $\Hull(T)$ and $T_c$ are distinct and $T$ is lesser simply paired.  

\begin{theorem*}[\cref{thm:3IsNotLSP,thm:When2CompLSP}]
  Let $T$ be a $C_{p^rq^s}$-transfer system. Then $(T,T_{c})$ and $(T,\Hull(T))$ are distinct and are the only compatible pairs with $T$ as the multiplicative part if and only if $T$ has two connected components, and the connected component of the trivial subgroup $e$ is either $\{ e, C_p, \dots, C_{p^r}\}$ or $\{e, C_q, \dots, C_{q^s}\}$.
  \end{theorem*}

\subsection*{Outline}

This paper is organised as follows. In \Cref{sec:background} we define transfer systems and saturated transfer systems and specify how we consider those notions for $C_{p^rq^s}$.  In \Cref{sec:pairs} we introduce compatible pairs, define lesser simply paired transfer systems, and prove basic results relating compatibility to saturation.  Finally, \Cref{sec:LSPOnGrid} is occupied with our main result, namely the characterisation of lesser simply paired transfer systems on $C_{p^rq^s}$.

\subsection*{Acknowledgements} We thank the organizers of the Women in Topology IV Workshop and the Hausdorff Institute for Mathematics in Bonn for providing this opportunity for collaborative research. We thank Scott Balchin, David Barnes, David Chan and Mike Hill for insightful conversations. We thank the anonymous referee for their valuable suggestions and comments. Osorno was partially supported by NSF Grant DMS-2204365. Van Niel was partially supported by NSF Grants DMS-RTG 2135960 and DMS-RTG 2135884. Zapata Castro was partially funded by NSF Grant DMS-1906281. Mazur's travel was supported by Elon University and Santhanam's travel was supported by the Indian Institute of Technology Bombay.

\section{Transfer Systems- Background and Notation}\label{sec:background}

We begin this section by defining a  $G$-transfer system, its connected components, and its saturated hull for any finite group $G$. We then transition to a discussion of $\C$-transfer systems.

\begin{definition}\label{def:tranfsys} Let $G$ be a finite group. A $G$-\emph{transfer system} $T$ is a partial order relation  on the set of subgroups of $G$, represented by edges
$\to$, that satisfies the following conditions.
\begin{enumerate}
    \item (Subgroup) $K \to H$ implies $K \leq  H$,
    \item(Reflexivity) $H\to H$ for all $H\leq G$,
    \item (Transitivity) $L\to K$ and $K\to H$   implies $L\to H$,
    \item (Restriction) $K\to H$ and $L\leq G$ implies $(K\cap L)\to (H\cap L)$,
    \item (Conjugation) $K\to H$ implies $gKg^{-1}\to gHg^{-1}$ for all $g\in G$.
\end{enumerate}
\end{definition}

We represent  $G$-transfer systems as directed graphs with vertices as subgroups and edges as the partial order relation, and thus we use the terms vertices and subgroups interchangeably. Moreover, we omit drawing  edges that represent the reflexivity condition.

A basic example of a $G$-transfer system is the full subgroup lattice of $G$, which is the transfer system that has all possible relations.

 \begin{definition}[Complete Transfer System, $T_c$]\label{def:CompleteTS}
We call the $G$-transfer system that has all possible edges the \emph{complete} transfer system, denoted $T_c$. 
\end{definition}

\begin{example}\label{ex:CompleteTS}
The following represents the complete $C_{p^2q}$-transfer system.

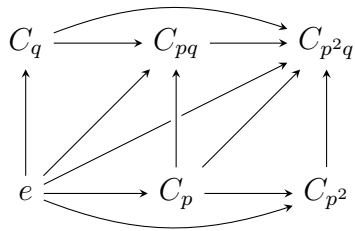
\centering{
\begin{tikzpicture}[>=stealth, bend angle=20, baseline=(current bounding box.center)]
    \node (Cq) at (0,0) {$C_q$};
    \node (Cpq) at (2,0) {$C_{pq}$};
    \node (Cp2q) at (4,0) {$C_{p^2q}$};
    \node (e) at (0,-2) {$e$};
    \node (Cp) at (2,-2) {$C_p$};
    \node (Cp2) at (4,-2) {$C_{p^2}$};

    \draw[->] (Cq) -- (Cpq);
    \draw[->, bend left] (Cq) to (Cp2q);
    \draw[->] (e) -- (Cq);
      \draw[->] (e) -- (Cpq);
    \draw[->, bend right] (e) to (Cp2);
    \draw[->] (e) to (Cp);
    \draw[->] (e) to (Cp2q);
    \draw[white, line width=4pt] (Cp) -- (Cpq);
    \draw[->] (Cp) to (Cpq);
    \draw[->] (Cp) to (Cp2);
    \draw[->] (Cp) to (Cp2q);
    \draw[->] (Cp2) to (Cp2q);
    \draw[->] (Cpq) to (Cp2q);

    \draw[white] (-0.6,-2.6)--(4.6,-2.6)--(4.6,0.6)--(-0.6,0.6)--(-0.6,-2.6);
\end{tikzpicture}}

\end{example}
 
\begin{example}\label{firstexample} The diagram below shows a more interesting $C_{p^2q}$-transfer system.

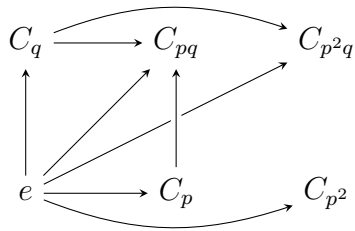
\centering{
\begin{tikzpicture}[>=stealth, bend angle=20, baseline=(current bounding box.center)]
    \node (Cq) at (0,0) {$C_q$};
    \node (Cpq) at (2,0) {$C_{pq}$};
    \node (Cp2q) at (4,0) {$C_{p^2q}$};
    \node (e) at (0,-2) {$e$};
    \node (Cp) at (2,-2) {$C_p$};
    \node (Cp2) at (4,-2) {$C_{p^2}$};

    \draw[->] (Cq) -- (Cpq);
    \draw[->, bend left] (Cq) to (Cp2q);
    \draw[->] (e) -- (Cq);
      \draw[->] (e) -- (Cpq);
    \draw[->, bend right] (e) to (Cp2);
    \draw[->] (e) to (Cp);
    \draw[->] (e) to (Cp2q);
    \draw[white, line width=4pt] (Cp) -- (Cpq);
    \draw[->] (Cp) to (Cpq);

    \draw[white] (-0.6,-2.6)--(4.6,-2.6)--(4.6,0.6)--(-0.6,0.6)--(-0.6,-2.6);
    
\end{tikzpicture}}
\end{example}

Drawing a $G$-transfer system as a graph allows us to consider its \textit{connected components}, which we define as the connected components of the underlying undirected graph.  

\begin{definition}
    We say that two vertices in a $G$-transfer system are in the same \emph{connected component} if there is an undirected path from one to the other.
\end{definition}

\begin{example} The transfer system in \Cref{firstexample} has  one connected component. In \Cref{fig:ConnCompsEx} we give three examples of $C_{p^2q}$-transfer systems with their connected components outlined in blue. Note that, in the middle transfer system, $C_{pq}$ and $C_{p^2}$  are in the same connected component even though there is no edge between them. 
\end{example}

\begin{figure}[hbt!]
    \centerline{
    \resizebox{\textwidth}{!}{%
\begin{tikzpicture}[>=stealth, baseline=(current bounding box.center), bend angle=20]

\node (00a) at (-8,0){$e$};
\node (10a) at (-6,0){$C_p$};
\node (20a) at (-4,0){$C_{p^2}$};
\node (01a) at (-8,2){$C_q$};
\node (11a) at (-6,2){$C_{pq}$};
\node (21a) at (-4,2){$C_{p^2q}$};
\node (00b) at (0,0){$e$};
\node (10b) at (2,0){$C_p$};
\node (20b) at (4,0){$C_{p^2}$};
\node (01b) at (0,2){$C_q$};
\node (11b) at (2,2){$C_{pq}$};
\node (21b) at (4,2){$C_{p^2q}$};
\node (00c) at (8,0){$e$};
\node (10c) at (10,0){$C_p$};
\node (20c) at (12,0){$C_{p^2}$};
\node (01c) at (8,2){$C_q$};
\node (11c) at (10,2){$C_{pq}$};
\node (21c) at (12,2){$C_{p^2q}$};
\draw[->](00a) to (10a);
\draw[->] (00a) to (01a);
\draw[->, bend right] (00a) to (20a);
\draw[->] (10a) to (20a);
\draw[->] (11a) to (21a);

\draw[->] (00b) to (01b);
\draw[->] (10b) to (11b);
\draw[->] (10b) to (21b);
\draw[->] (10b) to (20b);

\draw[->] (00c) to (01c);
\draw[->] (10c) to (20c);
\draw[blue] (-8.6,-0.6)--(-8.6,2.6)--(-7.4,2.6)--(-7.4,0.6)--(-3.6,0.6)--(-3.6,-0.6)--(-8.6,-0.6);
\draw[blue] (-6.6,2.6)--(-3.6,2.6)--(-3.6,1.4)--(-6.6,1.4)--(-6.6,2.6);

\draw[blue] (-0.6,-0.6)--(-0.6,2.6)--(0.6,2.6)--(0.6,-0.6)--(-0.6,-0.6);
\draw[blue] (1.4,-0.6)--(1.4,2.6)--(4.6,2.6)--(4.6,-0.6)--(1.4,-0.6);
\draw[blue] (7.4,-0.6)--(7.4,2.6)--(8.6,2.6)--(8.6,-0.6)--(7.4,-0.6);
\draw[blue] (9.4,-0.6)--(9.4,0.6)--(12.6,0.6)--(12.6,-0.6)--(9.4,-0.6);
\draw[blue] (9.4,1.4)--(9.4,2.6)--(10.6,2.6)--(10.6,1.4)--(9.4,1.4);
\draw[blue] (11.4,1.4)--(11.4,2.6)--(12.6,2.6)--(12.6,1.4)--(11.4,1.4);
\end{tikzpicture}}}
    \caption{The connected components of three $C_{p^2q}$-transfer systems.}
    \label{fig:ConnCompsEx}
\end{figure}
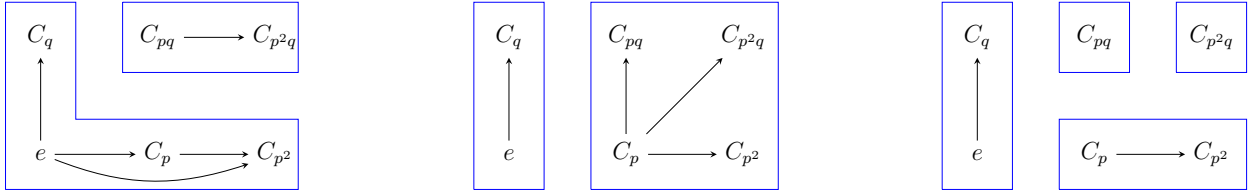

 Our results on compatible pairs rely heavily on the concept of \emph{saturation}, which is a special property of $G$-transfer systems
that was introduced by Rubin in \cite{RubinDetectingOperads}.  
 The \emph{saturation conjecture} stated in \cite{RubinDetectingOperads} provides a bridge between certain combinatorial properties of transfer systems and linear isometries operads. The saturation conjecture was proved for $C_{pq^s}$ in \cite{SaturatedWithUndergrads}, for $\C$ in \cite{Bannwart2023RealizationOS} and for cyclic and rank 2 groups in \cite{macbrough2023equivariant}. In addition, \cite{SaturatedWithUndergrads} enumerates the saturated transfer systems for $\C$.

\begin{definition}
A $G$-transfer system $T$ is \emph{saturated} if it satisfies the following ``two out of three" condition: if $L\leq K\leq H\leq G$, and $T$ contains two of the three edges $L\to K$, $L\to H$, $K\to H$, then $T$ contains the third as well. 
\end{definition}

By transitivity, if $T$ contains $L\to K$ and $K\to H$, then it must contain $L \to H$, and by restriction, if $T$ contains $L\to H$ then it must contain $L \to K$. Hence, we can rephrase the ``two out of three" condition by saying that if $L\leq K\leq H\leq G$ and $T$ contains $L\to H$, then $T$ must contain $K\to H$ as well. Colloquially, the ``two out of three condition" requires that $T$ includes all ``short" edges that sit inside ``long" edges.  See \Cref{fig:saturated}. 

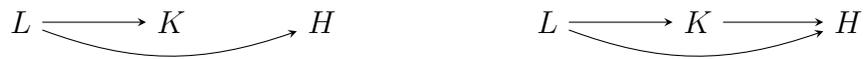
\begin{figure}[hbt!]
\centerline{
\begin{tikzpicture}[>=stealth, bend angle=20, baseline=(current bounding box.center)]
    \node (L1) at (0,0) {$L$};
    \node (K1) at (2,0) {$K$};
    \node (H1) at (4,0) {$H$};
    \node (L2) at (7,0) {$L$};
    \node (K2) at (9,0) {$K$};
    \node (H2) at (11,0) {$H$};
    \draw[->] (L1) to (K1);
    \draw[->, bend right] (L1) to (H1);

    \draw[->] (L2) to (K2);
    \draw[->, bend right] (L2) to (H2);
    \draw[->] (K2) to (H2);
\end{tikzpicture}}
\caption{The transfer system on the left is not saturated, while the transfer system on the right is.}
\label{fig:saturated}
\end{figure}

\begin{example}
    For any group $G$, the complete transfer system $T_c$ is saturated.
\end{example}

\begin{example}\label{secondexample}
    The  $C_{p^2q}$-transfer system shown below  is saturated.

    
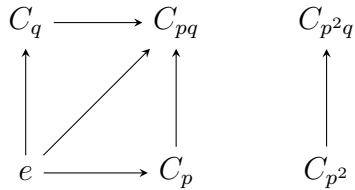
\centering
\begin{tikzpicture}[>=stealth, bend angle=20, baseline=(current bounding box.center)]
    \node (Cq) at (0,0) {$C_q$};
    \node (Cpq) at (2,0) {$C_{pq}$};
    \node (Cp2q) at (4,0) {$C_{p^2q}$};
    \node (e) at (0,-2) {$e$};
    \node (Cp) at (2,-2) {$C_p$};
    \node (Cp2) at (4,-2) {$C_{p^2}$};

    \draw[->] (e) to (Cq);
    \draw[->] (e) to (Cpq);
    \draw[->] (e) to (Cp);
    \draw[->, overlay] (Cp) to (Cpq);
    \draw[->] (Cp2) to (Cp2q);
    \draw[->] (Cq) -- (Cpq);

    \draw[white] (-0.6,-2.6)--(4.6,-2.6)--(4.6,0.6)--(-0.6,0.6)--(-0.6,-2.6);
\end{tikzpicture}
   
\end{example}

The transfer system in \cref{firstexample} is not saturated. Concretely, when we let $L=e$, $K=C_{p^2}$ and $H=C_{p^2q}$, the edge $K\to H$ is missing. Nevertheless, we can think about adding the minimum number of edges needed to make a transfer system saturated. This inspires the following definition.

\begin{definition}[Saturated Hull, $\Hull(T)$]\label{Def:SatHull}
    Let $T$ be a $G$-transfer system. The \emph{saturated hull} of $T$, $\Hull(T)$, is the smallest saturated $G$-transfer system that contains $T$.
\end{definition}

Note that the saturated hull of a transfer system $T$ always exists, and we can  construct it explicitly by taking the intersection of all  saturated transfer systems containing $T$.

\begin{example}
    The saturated hull of the transfer system in \cref{firstexample} is the complete transfer system shown in \Cref{ex:CompleteTS}. 
    \end{example}

\begin{example}\label{ex:SatHullEx}    
     \Cref{secondexample} is the saturated hull of the transfer system shown below.
    
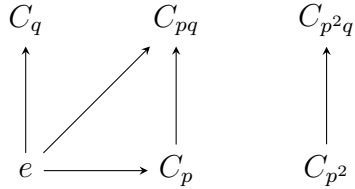
\[
    \begin{tikzpicture}[>=stealth, bend angle=20, baseline=(current bounding box.center)]
    \node (Cq) at (0,0) {$C_q$};
    \node (Cpq) at (2,0) {$C_{pq}$};
    \node (Cp2q) at (4,0) {$C_{p^2q}$};
    \node (e) at (0,-2) {$e$};
    \node (Cp) at (2,-2) {$C_p$};
    \node (Cp2) at (4,-2) {$C_{p^2}$};

    \draw[->] (e) to (Cq);
    \draw[->] (e) to (Cpq);
    \draw[->] (e) to (Cp);
    \draw[->] (Cp) to (Cpq);
    \draw[->] (Cp2) to (Cp2q);

    \draw[white] (-0.6,-2.6)--(4.6,-2.6)--(4.6,0.6)--(-0.6,0.6)--(-0.6,-2.6);
\end{tikzpicture}
    \]
\end{example} 

There is a nice relationship between the connected components of a transfer system and the connected components of its saturated hull. In addition to playing a crucial role in the results of \Cref{sec:LSPOnGrid}, this relationship tells us exactly when the saturated hull of a transfer system is the complete transfer system. We state the relationship in \Cref{fact:HullCompsComplete}, the proof of which requires the following lemma. (We  also use this lemma in the proof of \Cref{lem:SmallestIsSmallest}.)

\begin{lemma}\label{lem:zigzag}
Let $T$ be a $G$-transfer system. If two vertices are in the same connected component of $T$, then there is an undirected path of at most length two between them. 
\end{lemma}

\begin{proof}
Let $A$ and $B$ be vertices in the same connected component of $T$. First, if $T$ contains an undirected path of length two of the form $$A \rightarrow C \leftarrow B,$$ then by restriction, 
$T$ contains the edges $A \cap B \rightarrow A$ and $A \cap B \rightarrow B$. 
This means that $T$ now also has an undirected path from $A$ to $B$ with the direction of the arrows going in the opposite direction, i.e., $T$ contains the path
\[
A \leftarrow A \cap B \rightarrow B.
\]
By definition, this path still lies  in the same connected component as $A$ and $B$.

Now, assume that $T$ contains  an undirected path
\[
A \leftarrow A_1 \rightarrow A_2 \leftarrow A_3 \rightarrow B.
\]
(Note that arrows can also be the identity, so it does not matter if this path starts or ends with a right or a left arrow.) If we apply the above method to the middle part
\[
A_1 \rightarrow A_2 \leftarrow A_3,
\]
we get
\[
A_1 \leftarrow A_1 \cap A_3 \rightarrow A_3,
\]
and therefore the undirected path becomes
\[
A \leftarrow A_1 \leftarrow A_1 \cap A_3 \rightarrow A_3 \rightarrow B.
\]
Composing the arrows going in the same direction yields a path of length two. Applying this method inductively to an arbitrary undirected path in $T$ gives the desired result.\end{proof}

 \begin{proposition}\label{fact:HullCompsComplete} Let $T$ be a $G$-transfer system.  Two vertices are in the same connected component of $T$ if and only if they are in the same connected component of $\Hull(T)$. Moreover, if $K \le H$ and $K$ and $H$ are in the same connected component of $T$ then $\Hull(T)$ contains $K \to H$.

 \end{proposition}

 \begin{proof}  First, if two subgroups are in the same connected component of a $G$-transfer system $T$, then they are in the same connected component of $\Hull(T)$ because $\Hull(T)$ contains $T$.

 Conversely, given a $G$-transfer system $T$, we form $\Hull(T)$ by adding the edge $K \to H$ whenever $L \le K \le H$ and $T$ contains $L \to H$. Since $T$ already contains $L\to H$ and $L \to K$ (by restriction), adding $K\to H$ does not create a new connection between vertices. Hence, if two vertices are not in the same component of $T$, then they will not be in the same component of $\Hull(T)$. 

Finally, if $K\leq H$ and $K$ and $H$ are in the same component of $T$, there exists an undirected path of edges in $T$ connecting them both. 
By \Cref{lem:zigzag}, we can assume this path to be of the form $K \leftarrow L \rightarrow H$ for some $L \le K \le H$. In that case, the saturation condition implies precisely that $K\to H$ is in $\Hull(T)$. \end{proof}

\Cref{fact:HullCompsComplete} tells us that a transfer system and its saturated hull have the same number of connected components. It also tells us that the connected components of the hull are ``complete'' in the following sense: if $K \leq H$ are two subgroups in the same connected component of $T$, then there is an edge $K \rightarrow H$ in $\Hull(T)$. In particular, if a transfer system is connected, then so is its saturated hull. This leads to the following corollary.

\begin{corollary}\label{cor:connected}
Let $T$ be a $G$-transfer system. Then $\Hull(T)=T_c$ if and only if $T$ is connected. 
\end{corollary}

 \begin{example} Above we saw that the transfer system in \Cref{secondexample} is the saturated hull of the transfer system in \Cref{ex:SatHullEx}. Notice that both transfer systems have the same connected components, and each component of the saturated hull is complete.
 \end{example}

We devote the remainder of this section to $\C$-transfer systems, as these are the focus of our main results in \Cref{sec:LSPOnGrid}. We  display a $\C$-transfer system on a grid; we let the bottom left vertex of the grid   represent the trivial subgroup. Then moving to the right increases the power of $p$ and going up increases the power of $q$ so that the top right vertex represents $\C$. Thus, it is natural to use coordinate notation to refer to the vertices of a $\C$-transfer system. For example, the coordinate $(i,j)$ represents the subgroup $C_{p^iq^j}$. See \Cref{fig:grid}.

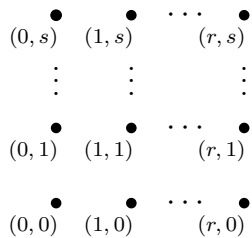
\begin{figure}[hbt!]
\centering
\begin{tikzpicture}
\fill (0,0) circle (2pt);
\fill (1,0) circle (2pt);
\fill (2.5,0) circle (2pt);

\fill (0,1) circle (2pt);
\fill (1,1) circle (2pt);
\fill (2.5,1) circle (2pt);

\filldraw (0,2.5) circle (2pt);
\fill (1,2.5) circle (2pt);
\fill (2.5,2.5) circle (2pt);

\node at (1.75,0) {$\cdots$};
\node at (1.75,1) {$\cdots$};
\node at (1.75,2.5) {$\cdots$};

\node at (0,1.75) {$\vdots$};
\node at (1,1.75) {$\vdots$};
\node at (2.5,1.75) {$\vdots$};

\tiny
\node at (-0.3,-0.3) {$(0,0)$};
\node at (0.7,-0.3) {$(1,0)$};
\node at (2.2,-0.3) {$(r,0)$};
\node at (-0.3,0.7) {$(0,1)$};
\node at (0.7,0.7) {$(1,1)$};
\node at (2.2,0.7) {$(r,1)$};
\node at (-0.3,2.2) {$(0,s)$};
\node at (0.7,2.2) {$(1,s)$};
\node at (2.2,2.2) {$(r,s)$};
\end{tikzpicture}
    \caption{ Visualizing of the subgroups of $\C$ as an $r \times s$ grid.}
    \label{fig:grid}
\end{figure}

Below we restate the definition of a $G$-transfer system specifically for $G=\C$ using the coordinate notation. Since the conjugation condition is trivial for abelian groups,  we omit it below.

\begin{definition}[$\C$-transfer system]\label{def:CtransSys}
    A \textit{$\C$-transfer system} $T$ is a partial order relation $\to$ on the set of vertices of the grid associated to $\C$ that satisfies the following conditions.

    \begin{enumerate}
        \item (Subgroup) If $(i_1,j_1)\to (i_2,j_2)$ then  $i_1\leq i_2$ and $j_1\leq j_2$.
        \item (Reflexivity) $(i,j)\to (i,j)$ for all  $(i,j)$.
        \item (Transitivity) If $(i_1,j_1)\to(i_2,j_2)$ and $(i_2,j_2)\to(i_3,j_3)$ then $(i_1,j_1)\to(i_3,j_3)$.
        \item (Restriction) If $(i_1,j_1)\to (i_2,j_2)$, then for any vertex $(a,b)$ we have  \newline $(\min\{i_1,a\},\min\{j_1,b\})\to (\min\{i_2,a\},\min\{j_2,b\})$.
    \end{enumerate}
\end{definition}

When the group is clear from context we display a $\C$-transfer system without labeling its vertices. For example, the diagram below shows the $C_{p^2q}$-transfer system from \Cref{firstexample}. 

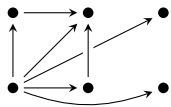
  \centerline{
\begin{tikzpicture}[>=stealth, baseline=(current bounding box.center),  bend angle=20]
\fill (0,0) circle (2pt);
\fill (1,0) circle (2pt);
\fill (2,0) circle (2pt);

\fill (0,1) circle (2pt);
\fill (1,1) circle (2pt);
\fill (2,1) circle (2pt);
\node (00) at (0,0){};
\node (10) at (1,0){};
\node (20) at (2,0){};
\node (01) at (0,1){};
\node (11) at (1,1){};
\node (21) at (2,1){};
\draw[->](00) to (10);
\draw[->] (00) to (11);
\draw[->] (00) to (01);
\draw[->] (00) to (21);
\draw[->, bend right] (00) to (20);
\draw[->] (01) to (11);
\draw[white, line width=4pt] (10) to (11);
\draw[->] (10) to (11);

\draw[white] (-0.3,-0.5)--(2.3,-0.5)--(2.3,1.3)--(-0.3,1.3)--(-0.3,-0.5);
\end{tikzpicture}}

\begin{notation}[The $\ij{i,j}_T$ component of a $\C$-transfer system]\label{not:component}
In \Cref{sec:LSPOnGrid}, given a $\C$-transfer system $T$, we will often want to consider  the connected component of $T$ that contains a specific vertex.   Hence, we denote  the connected component of the vertex $(i,j)$  by $\ij{i,j}_T$ or just $\ij{i,j}$ when the transfer system is clear. 
\end{notation}

\section{Compatible Pairs}\label{sec:pairs}

In this section we provide a formal definition for and  examples of compatible pairs of $G$-transfer systems. Then given a transfer system $T$ we prove that $(T,T_c)$ and $(T,\Hull(T))$ always form compatible pairs. This idea leads to asking when $T$ only forms these trivially compatible pairs, in which case we say that $T$ is \textit{lesser simply paired}. We end this section by examining compatibility specifically for $\C$-transfer systems.

\begin{definition}[{Compatible Pair, \cite[Definition 4.6]{ChanTambara}}]\label{def:CompatTSs}
    Let $T$ and $T'$ be $G$-transfer systems. Then $(T,T')$ is a \textit{compatible pair} if $T$ and $T'$ satisfy the following criteria. 
    \begin{enumerate}
        \item $T \subseteq T'$. 
        \item Suppose $A$, $B$, and $C$ are subgroups of $G$ such that $B$ and $C$ are subgroups of $A$. If $B \to A$ is in $T$ and $B\cap C \to B$ is in $T'$ then $C \to A$ must be in $T'$.
    \end{enumerate}
\end{definition}

The second criterion is best understood by the  diagram in \Cref{fig:CompatDiag}. Note that if $T$ contains the edge $B \to A$, then by restriction, $T$ contains $B\cap C\to C$ as well. Hence, the second criterion states that  if the black vertical edges  of \Cref{fig:CompatDiag} are in $T$ and the red horizontal edge is in $T'$, then $T'$  must contain $C \to A$ as well.

\begin{figure}[hbt!]
    \centerline{
\begin{tikzpicture}
    \node (A) at (1.5,2){$A$};
    \node (B) at (1.5,0){$B$};
    \node (BnC) at (-1.5,0){$B\cap C$};
    \node (C) at (-1.5,2){$C$};
    \draw[->] (BnC) -- (C) node[midway, left] {\tiny in $T$};
    \draw[->, blue, dashed] (C) -- (A) node[midway, above] {\tiny Need in $T'$};
    \draw[->] (B)--(A) node[midway, right] {\tiny in $T$};
    \draw[->, red] (BnC)--(B) node[midway, above] {\tiny in $T'$};
\end{tikzpicture}}
    \caption{Diagram for Criterion 2 of the  definition of compatible pairs (\Cref{def:CompatTSs})}
    \label{fig:CompatDiag}
\end{figure}
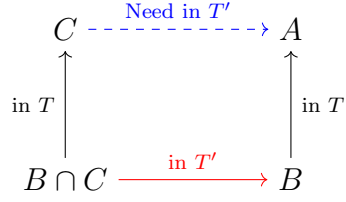

\begin{example}\label{ex:compatTS} The two transfer systems shown below are an example of a compatible pair of $C_{p^2q}$-transfer systems. The second criterion of \Cref{def:CompatTSs} (i.e., \Cref{fig:CompatDiag}) states that since these two are compatible and the transfer system on the right contains the edge $(0,0)\to (2,0)$, then it must contain the edge $(1,0)\to (2,1)$ as well.

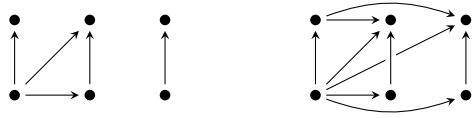
\[
\begin{tikzpicture}[>=stealth, baseline=(current bounding box.center), bend angle=20]
\fill (0,0) circle (2pt);
\fill (1,0) circle (2pt);
\fill (2,0) circle (2pt);

\fill (0,1) circle (2pt);
\fill (1,1) circle (2pt);
\fill (2,1) circle (2pt);

\fill (4,0) circle (2pt);
\fill (5,0) circle (2pt);
\fill (6,0) circle (2pt);

\fill (4,1) circle (2pt);
\fill (5,1) circle (2pt);
\fill (6,1) circle (2pt);

\node (00a) at (0,0){};
\node (10a) at (1,0){};
\node (20a) at (2,0){};
\node (01a) at (0,1){};
\node (11a) at (1,1){};
\node (21a) at (2,1){};
\node (00b) at (4,0){};
\node (10b) at (5,0){};
\node (20b) at (6,0){};
\node (01b) at (4,1){};
\node (11b) at (5,1){};
\node (21b) at (6,1){};
\draw[->] (00a) to (01a);
\draw[->] (00a) to (10a);
\draw[->] (00a) to (11a);
\draw[->] (10a) to (11a);
\draw[->] (20a) to (21a);

\draw[->] (00b) to (01b);
\draw[->] (00b) to (10b);
\draw[->, bend right] (00b) to (20b);
\draw[->, bend left] (01b) to (21b);
\draw[->] (01b) to (11b);
\draw[->] (00b) to (11b);
\draw[->] (00b) to (21b);
\draw[->] (20b) to (21b);
\draw[white, line width=4pt] (10b) to (11b);
\draw[->] (10b) to (11b);

\draw[white] (-0.3,-0.5)--(2.3,-0.5)--(2.3,1.3)--(-0.3,1.3)--(-0.3,-0.5);
\end{tikzpicture}
\]
\end{example}

Given a $G$-transfer system $T$, it is natural to ask which $G$-transfer systems that contain $T$ are compatible with $T$, but unless $G$ is $C_{p^n}$, this question is difficult to answer. (The case of $G=C_{p^n}$ is manageable because the second criterion of \Cref{def:CompatTSs} degenerates.)
However, \Cref{prop:SaturTCompatible} gives a class of transfer systems that are always compatible with $T$.

\begin{proposition}\label{prop:SaturTCompatible}
    Let $T$ be a $G$-transfer system. If $T'$ is a saturated $G$-transfer system that contains $T$ then $(T,T')$ is a compatible pair.
\end{proposition}

\begin{proof}
    Let $T'$ be a saturated transfer system. To show that $(T,T')$ is a compatible pair we must show that Criterion 2 of \Cref{def:CompatTSs} holds (i.e, \Cref{fig:CompatDiag}). So, let $A$, $B$, and $C$ be subgroups of $G$ such that $B$ and $C$ are subgroups of $A$. Further, assume $B \cap C \to B$ is in $T'$ and $B \to A$ is in $T$. We will argue that $T'$ contains $C \to A$. Since $T'$ contains $T$, it follows that $T'$ contains $B\cap C \to B$, $B \to A$, and $B\cap C \to C$. Then $T'$ contains $B\cap C \to A$ by transitivity. 
     Finally, since $T'$ is saturated and contains $B\cap C \to A $ and $B\cap C \to C$, it follows that $T'$ contains $C \to A$, see picture below.  \end{proof}

    
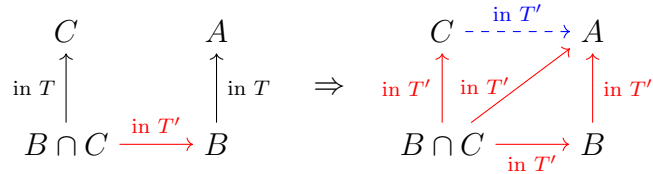
\[
    \begin{tikzpicture}
    \node (A) at (1,1.5){$A$};
    \node (B) at (1,0){$B$};
    \node (BnC) at (-1,0){$B\cap C$};
    \node (C) at (-1,1.5){$C$};
    \node (imp) at (2.5,0.75){$\Rightarrow$};
    \node (A1) at (6,1.5){$A$};
    \node (B1) at (6,0){$B$};
    \node (BnC1) at (4,0){$B\cap C$};
    \node (C1) at (4,1.5){$C$};
    \draw[->] (BnC) -- (C) node[midway, left] {\tiny in $T$};
    \draw[->] (B)--(A) node[midway, right] {\tiny in $T$};
    \draw[->, red] (BnC)--(B) node[midway, above] {\tiny in $T'$};
    \draw[->,red] (BnC1) -- (C1) node[midway, left] {\tiny in $T'$};
    \draw[->,red] (BnC1) -- (B1) node[midway, below] {\tiny in $T'$};
    \draw[->,red] (BnC1) -- (A1) node[midway, left] {\tiny in $T'$};
    \draw[->,red] (B1) -- (A1) node[midway, right] {\tiny in $T'$};
    \draw[->,blue, dashed] (C1) -- (A1) node[midway, above] {\tiny in $T'$};
    \end{tikzpicture}
    \]

\Cref{prop:SaturTCompatible} implies that every $G$-transfer system  is compatible with at least two transfer systems that contain it. 

\begin{corollary}\label{cor:hull}
    Let $T$ be a  $G$-transfer system. Then the following are compatible pairs.
    \begin{itemize} 
    \item $(T,\Hull(T))$ where $\Hull(T)$ is the saturated hull of $T$ (\Cref{Def:SatHull})
    \item $(T,T_c)$ where $T_c$ is the complete transfer system (\Cref{def:CompleteTS})
    \end{itemize}
\end{corollary}

Since $(T,\Hull(T))$ and $(T,T_c)$ are  compatible pairs for all $T$, we
 consider these  to be the ``trivially" compatible pairs.  Recall from \Cref{cor:connected} that $\Hull(T)=T_c$ if and only if $T$ has exactly one connected component.  Moreover,  the proposition below shows that $\Hull(T)$ is the smallest transfer system that is compatible with $T$.

 \begin{proposition}
 \label{prop:BigTContainsHull}
    Let $T$ and $T'$ be $G$-transfer systems. If $(T,T')$ is a compatible pair, then $T'$ contains $\Hull(T)$. 
\end{proposition}

\begin{proof}
    We recover the saturation condition from the definition of compatibility by letting $B = B \cap C$ in the diagram of \Cref{fig:CompatDiag}. This implies that $\Hull(T) \subseteq T'$.  
\end{proof}

Given a $G$-transfer system $T$ and a collection $S$ of edges, there exists a smallest transfer system $T'$ containing $S$ such that $(T,T')$ is a compatible pair. Indeed, $T'$ is the intersection of all transfer systems that contain $S$ and form a compatible pair with $T$. This intersection exists because the complete transfer system satisfies these two conditions.  In the following lemma we give an explicit way to construct $T'$ for a particular choice of $S$. This construction  plays a key role in \Cref{sec:LSPOnGrid}, see, in particular, \Cref{rem:HowToShowNotLSP}.

\begin{lemma}\label{lem:smallestcompatible}
    Let $T$ be a $G$-transfer system and let $S$ be a collection of edges with source $e$. Further, let $T'$ be the smallest  $G$-transfer system that contains $S$ and that is compatible with $T$. Then $T'$ can be constructed by completing the following four steps in order. 
    \begin{enumerate}
        \item Take closure of $T\cup S$ under restriction.
        \item Take the closure of the collection created in (1) under transitivity.
        \item Take the closure of the collection created in (2) under conjugation.
        \item  Take the closure of the  collection created in (3) under  compatibility  with respect to $T$.
    \end{enumerate}
\end{lemma}

\begin{proof}

Let $T_0$  be the closure of $T\cup S$ under restriction and let $T_1$ be the closure of $T_0$ under transitivity. Then $T_1$ is closed under restriction, since the restriction of a composition of edges can be written as a composition of their restrictions. 
    
    Now we take the closure of $T_1$ under conjugation and denote it by $T_2$. This is closed under restriction, since the restriction of the conjugate of an edge in $T_1$ is equal to the conjugate of the restriction of that same edge, and $T_1$ is closed under restriction. Similarly, $T_2$ is closed under transitivity since $T_1$ is closed under transitivity and the composition of conjugates of edges in $T_1$ is the conjugate of their composition.  It follows that $T\subset T_2$, and  $T_2$ is the smallest transfer system that contains $T\cup S$. In summary,  $T_2$ is the transfer system that results after doing Step (1), Step (2) and Step (3). Further, every edge in $T_2 \setminus T$ (i.e., every edge added during Steps (1), (2), and (3)) has source $e$.

    Let $T'$ be the closure of $T_2$ under the compatibility condition with respect to $T$, i.e., $T'$ is what we get from applying Step (4) to $T_2$. We will show that $T'$ is closed under restriction, composition and conjugation, and therefore that  $T'$ is  the smallest $G$-transfer system compatible with $T$ that contains $S$, which is our main claim.

    \begin{figure}[hbt!]
    \centerline{
\begin{tikzpicture}
    \node (A) at (1.5,2){$A$};
    \node (B) at (1.5,0){$B$};
    \node (BnC) at (-1.5,0){$B\cap C$};
    \node (C) at (-1.5,2){$C$};
    \draw[->] (BnC) -- (C) node[midway, left] {\tiny in $T$};
    \draw[->, blue, dashed] (C) -- (A) node[midway, above] {\tiny in $T' \backslash T_2$};
    \draw[->] (B)--(A) node[midway, right] {\tiny in $T$};
    \draw[->, red] (BnC)--(B) node[midway, above] {\tiny in $T_2$};
\end{tikzpicture}}
    \caption{The new edges of $T'$ obtained via compatibility.}
    \label{fig:CompatDiag2}
\end{figure}
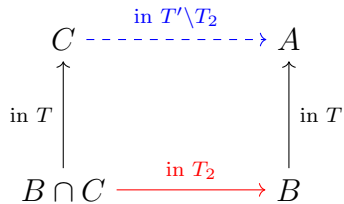

    We first show that all  edges in $T'$ arise from compatibility as  the top horizontal edge in \Cref{fig:CompatDiag2}, 
    where the vertical edges are in $T$ and the bottom horizontal edge is in $T_2$. This includes the edges in $T_2$, since they arise from the trivial compatibility diagrams where the vertical edges are the identity. Since we are considering compatibility with respect to $T$, the compatibility diagrams we consider will have vertical edges in $T$, and we will say that the top horizontal edge is obtained from the compatibility condition on the bottom horizontal edge.

         \begin{figure}[hbt!]
     \definecolor{nicegreen}{HTML}{009B55}
     \[
 \begin{tikzpicture}[>=stealth, baseline=(current bounding box.center), bend angle=20]
 \node (00ab) at (1,-3.5){$B\cap C$};
 \node (20ab) at (4,-3.5){$B$};
 \node (02ab) at (1,-2){$C$};
 \node (22ab) at (4,-2){$A$};
 \node (n02ab) at (1,-5){$B\cap C\cap E$};
 \node (n22ab) at (4,-5){$E$};
 \draw[->] (00ab) -- (02ab);
 \draw[->] (20ab) -- (22ab);
 \draw[->] (n02ab) -- (00ab);

 \draw[->, nicegreen] (n02ab) -- (n22ab) node[midway, below]{\tiny in $T_2$};

 \draw[->, magenta] (02ab) to (22ab);
 \draw[->, magenta] (00ab) -- (20ab);

 \draw[white, line width=4pt] (n22ab) to (20ab);
 \draw[->] (n22ab) -- (20ab);

 \end{tikzpicture}
 \]
         \caption{Compatibility can be done in one step.}
         \label{fig:compcomp}
    \end{figure}
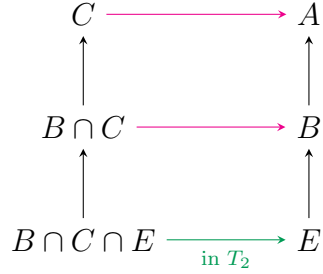
   
    To show that iterating Step (4) does not give us any new edges, consider \Cref{fig:compcomp}. There, the edge $C \to A$ is obtained from the compatibility condition on an  edge that was itself obtained from the compatibility condition on an edge in $T_2$. Since $T$ is closed under transitivity, the edges $B\cap C \cap E \to C$ and $E \to A$ are both in $T$. Given that $C\cap B \cap E = C \cap E$, the outer square in  \Cref{fig:compcomp} is a compatibility diagram, and it follows that $C \to A$ is given directly by the compatibility condition on an edge in $T_2$. Therefore, iterating Step (4) does not give us any new edges, and so to show that $T'$ is a transfer system it is sufficient to only consider edges that are either in $T_2$ or that arise via the compatibility condition as in \Cref{fig:CompatDiag2}.

To show that $T'$ is closed under restriction, let $C\to A $ be an edge in $T'$, and let $D$ be a subgroup of $G$. To show that $T'$ contains $C\cap D \to A \cap D$, we restrict the diagram in \Cref{fig:CompatDiag2} along $D$. Since $T$ and $T_2$ are transfer systems, the vertical edges of the resulting diagram are in $T$ and the bottom horizontal edge is contained in $T_2$. Similarly, we show that $T'$ is closed under conjugation by conjugating the compatibility diagram in \Cref{fig:CompatDiag2} by $g$ for all $g\in G$.

   It remains to show that $T'$ is closed under transitivity.  Given  two edges $C \to A $ and $A \to D $ in $T'$
we consider  \Cref{fig:comp1}.

 \begin{figure}[hbt!]
    \definecolor{nicegreen}{HTML}{009B55}
    \[
\begin{tikzpicture}[>=stealth, baseline=(current bounding box.center), bend angle=17]
\node (00a) at (-3,-1){$C\cap B$};
\node (20a) at (-1,-1){$B$};
\node (02a) at (-3,1){$C$};
\node (22a) at (-1,1){$A$};
\node (00b) at (1,-1){$A\cap E$};
\node (20b) at (3,-1){$E$};
\node (02b) at (1,1){$A$};
\node (22b) at (3,1){$D$};
\node (00ab) at (6,0){$C\cap E$};
\node (20ab) at (9,0){$A\cap E$};
\node (40ab) at (11,0){$E$};
\node (02ab) at (6,1.5){$C\cap A=C$};
\node (22ab) at (9,1.5){$A$};
\node (42ab) at (11,1.5){$D$};
\node (n02ab) at (6,-1.5){$C\cap B\cap E$};
\node (n22ab) at (9,-1.5){$B\cap E$};
\draw[->] (00a) -- (02a) node[midway, left]{\tiny in $T$};
\draw[->] (20a) -- (22a) node[midway, right]{\tiny in $T$};
\draw[->] (00b) -- (02b) node[midway, left]{\tiny in $T$};
\draw[->] (20b) -- (22b) node[midway, right]{\tiny in $T$};
\draw[->] (00ab) -- (02ab);
\draw[->] (20ab) -- (22ab);
\draw[->] (40ab) -- (42ab);
\draw[->] (n02ab) -- (00ab);

\draw[->, nicegreen] (00a) -- (20a) node[midway, below]{\tiny in $T_2$};
\draw[->, nicegreen] (00b) -- (20b)
node[midway, below]{\tiny in $T_2$};
\draw[->, nicegreen] (20ab) to (40ab);
\draw[->, nicegreen] (n02ab) -- (n22ab) node[midway, below]{\tiny in $T_2$};

\draw[->, magenta] (02a) -- (22a)
node[midway, above]{\tiny in $T'$};
\draw[->, magenta] (02b) -- (22b)
node[midway, above]{\tiny in $T'$};
\draw[->, magenta] (02ab) to (22ab);
\draw[->, magenta] (22ab) to (42ab);
\draw[->, magenta] (00ab) -- (20ab) node[midway, above]{\tiny in $T'$};

\draw[->, blue, bend left] (02ab) to (42ab);
\draw[->, magenta, bend right] (00ab) to (40ab);

\draw[white, line width=4pt] (n22ab) to (20ab);
\draw[->] (n22ab) -- (20ab) node[midway, right]{\tiny in $T$};

\end{tikzpicture}
\]
        \caption{Composition of edges in $T'$ }
        \label{fig:comp1}
    \end{figure}
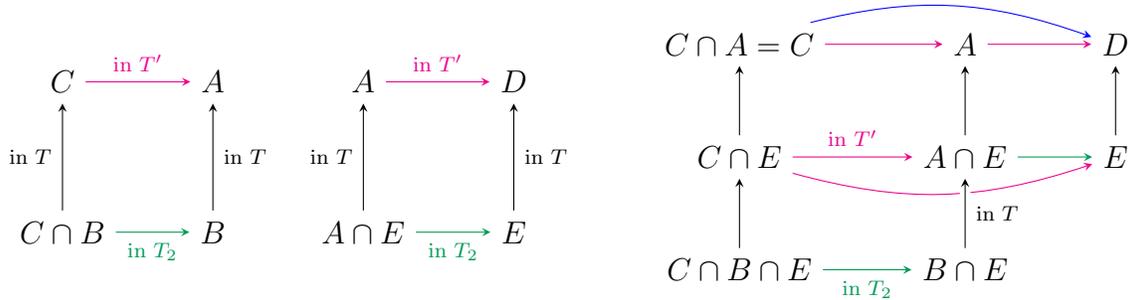

In \Cref{fig:comp1}, let $C \to A$ and $A \to D$ be obtained in $T'$ by the compatibility condition on the edges $C\cap B \to B$ and $A\cap E \to E$ in $T_2$, respectively. Then consider the diagram on the right of the figure, where the bottom square is obtained by restricting the compatibility diagram for $C\to A$ along $E$ and the edge $C\cap E \to C$ is obtained by restricting $A\cap E \to A$ along $C$. If $A \cap E$ is not the trivial subgroup, then $A\cap E \to E$ must be in $T$, since all the edges in $T_2\setminus T$ have the trivial group as their source. By transitivity of $T$ we then have that $B\cap E \to D$ is in $T$, which gives a compatibility diagram for $C \to D$ with bottom corners $C\cap B \cap E$ and $B\cap E$. If instead  $A \cap E$ is the trivial subgroup, then so is $C\cap E$, and in that case $C \cap E \to E$ is in $T_2$, and we obtain $C \to D$ by the compatibility condition on $C \cap E \to E$. In either case, it follows that $C\to D$ it in $T'$, as desired. 

Therefore, $T'$ is a $G$-transfer system and is the smallest $G$-transfer system compatible with $T$ that contains $S$.

\end{proof}

Now that we established a method how to construct compatible pairs in general,
let us return to the more concrete study of compatible pairs of $\C$-transfer systems. A special case of this is $C_{p^n}$, which was examined by Hill, Meng, and Li in \cite{Hill-Meng-Nan}.
They show that since the lattice is totally ordered,  the compatibility condition is equivalent to the saturation condition. Hence, if $T$ and $T'$ are $C_{p^n}$-transfer systesm with $T'$ containing $\Hull(T)$, then $(T,T')$ forms a compatible pair. In other words, when $G=C_{p^n}$, the converse of \Cref{prop:BigTContainsHull} holds. 
However, the converse does not hold in general as we demonstrate in the following example. 
\begin{example}
The diagrams below show a pair of $C_{pq}$-transfer systems $T\subset T'$ such that $\Hull{(T)}=T$ (and thus $T'$ contains $\Hull(T)$) but $(T,T')$ is not a compatible pair. To see  that the compatibility condition fails, let  $A=(1,1)$, $B=(0,1)$ and $C=(1,0)$. 
\[
\begin{tikzpicture}[>=stealth, baseline=(current bounding box.center), bend angle=20]
\fill (0,0) circle (2pt);
\fill (1,0) circle (2pt);

\fill (0,1) circle (2pt);
\fill (1,1) circle (2pt);

\fill (4,0) circle (2pt);
\fill (5,0) circle (2pt);

\fill (4,1) circle (2pt);
\fill (5,1) circle (2pt);

\node (00a) at (0,0){};
\node (10a) at (1,0){};
\node (01a) at (0,1){};
\node (11a) at (1,1){};

\node (00b) at (4,0){};
\node (10b) at (5,0){};

\node (01b) at (4,1){};
\node (11b) at (5,1){};

\draw[->] (00a) to (10a);
\draw[->] (01a) to (11a);

\draw[->] (00b) to (01b);
\draw[->] (01b) to (11b);
\draw[->] (00b) to (10b);
\draw[->] (00b) to (11b);

\draw[white] (-0.3,-0.5)--(2.3,-0.5)--(2.3,1.3)--(-0.3,1.3)--(-0.3,-0.5);
\end{tikzpicture}
\]
\end{example}

\begin{remark} \label{remark:HML}
In \cite{Hill-Meng-Nan}, Hill, Meng and Li  went one step further for $C_{p^n}$-transfer systems, computing the number of compatible pairs of transfer systems on $C_{p^n}$ to be the Fuss-Catalan number

$$A_{n+1}(3,1)=\frac{1}{3n+1}\binom{3n+1}{n}.$$ 

A vital ingredient for this enumeration is the notion of the \emph{core} of a $C_{p^n}$-transfer system $T$, which is defined as the sub-transfer system of $T$ generated by all of the edges of $T$ of the form $i \rightarrow i+1$. Hill, Meng, and Li show that the core is the largest sub-transfer system of $T$ that can be expressed as a disjoint union of complete transfer systems. Proposition 3.2 of \cite{Hill-Meng-Nan} then states that a pair of $C_{p^n}$-transfer systems $T \subseteq T'$ is compatible if and only if $\Hull(T) \subseteq \text{Core}(T').$  This idea forms the basis for a recursion formula leading to their enumeration result for $C_{p^n}$. Unfortunately, this classification cannot be extended to groups other than $C_{p^n}$ (i.e, $\C$) because there is no good definition of ``core" for more complicated groups.
\end{remark}

  In this paper, we focus on identifying $\C$-transfer systems that form only trivially compatible pairs in the sense of the following definition.

\begin{definition}[Lesser Simply Paired]
We say that a $G$-transfer system $T$ is \textit{lesser simply paired }   if, for all $T'$ such that $T\subseteq T'$, $(T,T')$ is a compatible pair if and only if $T'$ is $\Hull(T)$ or the complete transfer system $T_c$.
\end{definition}

\begin{example}\label{ex:todo} The  $C_{p^2q}$-transfer system on the left below is lesser simply paired. Below we provide an easy way of deducing this.  The  transfer system on the right is not lesser simply paired. Indeed, in \Cref{ex:compatTS} we showed that it is compatible with a transfer system that is neither its saturated hull nor the complete transfer system.

\begin{center}
\[
\begin{tikzpicture}[>=stealth, baseline=(current bounding box.center),  bend angle=20]
\fill (0,0) circle (2pt);
\fill (1,0) circle (2pt);
\fill (2,0) circle (2pt);

\fill (0,1) circle (2pt);
\fill (1,1) circle (2pt);
\fill (2,1) circle (2pt);
\node (00) at (0,0){};
\node (10) at (1,0){};
\node (20) at (2,0){};
\node (01) at (0,1){};
\node (11) at (1,1){};
\node (21) at (2,1){};
\draw[->](00) to (10);
\draw[->] (00) to (11);
\draw[->] (00) to (01);
\draw[->] (00) to (21);
\draw[->, bend right] (00) to (20);
\draw[->] (01) to (11);
\draw[white, line width=4pt] (10) to (11);
\draw[->] (10) to (11);

\fill (5,0) circle (2pt);
\fill (6,0) circle (2pt);
\fill (7,0) circle (2pt);

\fill (5,1) circle (2pt);
\fill (6,1) circle (2pt);
\fill (7,1) circle (2pt);

\node (00a) at (5,0){};
\node (10a) at (6,0){};
\node (20a) at (7,0){};
\node (01a) at (5,1){};
\node (11a) at (6,1){};
\node (21a) at (7,1){};

\draw[->] (00a) to (01a);
\draw[->] (00a) to (10a);
\draw[->] (00a) to (11a);
\draw[->] (10a) to (11a);
\draw[->] (20a) to (21a);

\draw[white] (-0.3,-0.5)--(2.3,-0.5)--(2.3,1.3)--(-0.3,1.3)--(-0.3,-0.5);
\end{tikzpicture}
\]
\end{center}

\end{example}

Combining \Cref{cor:connected} with \Cref{prop:BigTContainsHull} immediately gives us a class of lesser simply paired transfer systems.

\begin{theorem}\label{thm:OneCompLSP}
    Let $T$ be a $G$-transfer system. If $T$ has exactly one connected component, then $T$ is lesser simply paired. 
\end{theorem}
\begin{proof}
By \Cref{cor:connected}, since $T$ has only one component it follows that $\Hull(T)$ is the complete transfer system. Since \Cref{prop:BigTContainsHull} shows that if $(T,T')$ is a compatible pair then $T'$ contains $\Hull(T)$, it follows that $T$ is lesser simply paired.
\end{proof}

 In the remainder of this paper, we rely heavily on understanding the definition of compatibility specifically for $\C$-transfer systems. So, we end this section with a review of the definition in the context of the $(r\times s)$-grid.

\begin{remark}[Compatibility of $\C$-Transfer Systems]\label{rem:CpqCompDiags} Given two $\C$-transfer systems $T$ and $T'$ with $T$ contained in $T'$, to determine if $(T,T')$ is a compatible pair, by \Cref{prop:BigTContainsHull}, we first check if $T'$ contains $\Hull(T)$. Then $T$ and $T'$ must satisfy the  compatibility diagram shown in \Cref{fig:CompatDiag} for any subgroups $A$, $B$, and $C$ such that $B$ and $C$ are subgroups of $A$.

When $B = C$ the diagram is trivial. When $B \ne C$ there are three cases to consider.
\begin{itemize}

\item \textbf{When $C<B$:} Since $B \cap C = C$ and $T'$ contains $T$, by the transitivity condition, $T'$ contains the blue dashed edge $C \to A$ as well. Hence, the compatibility diagram follows trivially from the definition of transfer system.

\item \textbf{When $B < C$:} 
In this case, $B \cap C =B$, so the above square translates to the following: if edges $B \rightarrow C$ and $B \rightarrow A$ are both in $T$, then $C \rightarrow A$ has to be in $T'$. This  spells out  that $T'$ must contain $\Hull(T)$, which we already covered in \Cref{prop:BigTContainsHull}.

\item \textbf{When $B$ and $C$ are not comparable:} This is the only case where the  definition of compatibility gives conditions beyond \Cref{prop:BigTContainsHull}. Therefore, when we use the definition of compatibility in future proofs, we may assume $B$ and $C$ are not comparable. Further, when $B$ and $C$ are not comparable they lie in different rows/columns in the subgroup lattice of $\C$, hence in future proofs we need only consider compatibility diagrams of the forms shown in \Cref{fig:CpqCompatDiags}.

\begin{figure}[hbt!]
    \centering
   \begin{tikzpicture}[>=stealth, baseline=(current bounding box.center), bend angle=20]
\fill (0,0) circle (2pt);
\fill (1,0) circle (2pt);
\fill (2,0) circle (2pt);
\fill (3,0) circle (2pt);
\fill (4,0) circle (2pt);

\fill (0,1) circle (2pt);
\fill (1,1) circle (2pt);
\fill (2,1) circle (2pt);
\fill (3,1) circle (2pt);
\fill (4,1) circle (2pt);

\fill (0,2) circle (2pt);
\fill (1,2) circle (2pt);
\fill (2,2) circle (2pt);
\fill (3,2) circle (2pt);
\fill (4,2) circle (2pt);

\fill (0,3) circle (2pt);
\fill (1,3) circle (2pt);
\fill (2,3) circle (2pt);
\fill (3,3) circle (2pt);
\fill (4,3) circle (2pt);

\fill (6,0) circle (2pt);
\fill (7,0) circle (2pt);
\fill (8,0) circle (2pt);
\fill (9,0) circle (2pt);
\fill (10,0) circle (2pt);

\fill (6,1) circle (2pt);
\fill (7,1) circle (2pt);
\fill (8,1) circle (2pt);
\fill (9,1) circle (2pt);
\fill (10,1) circle (2pt);

\fill (6,2) circle (2pt);
\fill (7,2) circle (2pt);
\fill (8,2) circle (2pt);
\fill (9,2) circle (2pt);
\fill (10,2) circle (2pt);

\fill (6,3) circle (2pt);
\fill (7,3) circle (2pt);
\fill (8,3) circle (2pt);
\fill (9,3) circle (2pt);
\fill (10,3) circle (2pt);

\tiny
\node at (4.2,3.2){$A$};
\node (Al) at (4,3){};
\node at (3.2,0.8){$C$};
\node (Cl) at (3,1){};
\node at (0.6,0.8){$B\cap C$};
\node (BnCl) at (1,1){};
\node at (0.8,2.2){$B$};
\node (Bl) at (1,2){};

\node at (6.8,2.2){$C$};
\node (Cr) at (7,2){};
\node at (6.6,-0.2) {$B\cap C$};
\node (BnCr) at (7,0){};
\node at (9.2,-0.2) {$B$};
\node (Br) at (9,0){};
\node at (10.2,3.2) {$A$};
\node (Ar) at (10,3){};

\draw[->, red] (BnCl) -- (Bl) node[midway,left]{in $T'$};
\draw[->, bend right] (BnCl) to (Cl);
\draw[->] (Bl) -- (Al) node[midway, above]{in $T$};
\draw[->, blue, dashed] (Cl) -- (Al) node[above=-17mm]{need in $T'$};
\node at (2,0.6){in $T$};

\draw[->, bend left] (BnCr) to (Cr);
\node at (6.4,1.5) {in $T$};
\draw[->,red, bend right] (BnCr) to (Br);
\node[red] at (8,-0.4){in $T'$};
\draw[->, blue, dashed] (Cr) -- (Ar); \node[blue] at (8,2.7){need in $T'$};
\draw[->] (Br)--(Ar) node[midway, right]{in $T$};
\end{tikzpicture}
    \caption{Compatibility visualised for $\C$-transfer systems and the underlying grid. Note that $A$ is an arbitrary vertex and not necessarily $(r,s)$.}
    \label{fig:CpqCompatDiags}
\end{figure}
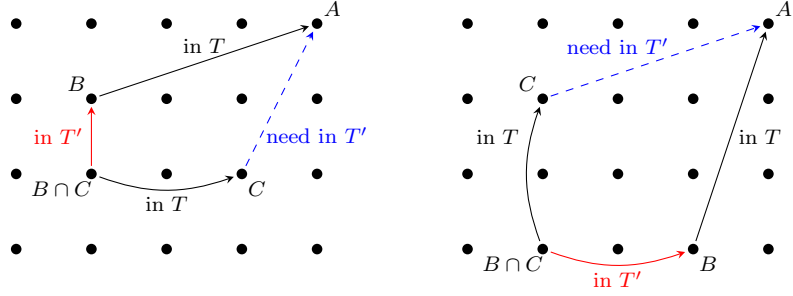

\end{itemize}

\end{remark}

\section{Identifying Lesser Simply Paired $\C$-Transfer Systems}\label{sec:LSPOnGrid}

This section focuses specifically on $\C$-transfer systems, and our main goal is to determine which $\C$-transfer systems are lesser simply paired. In other words, we determine all $\C$-transfer systems $T$ such that the only larger transfer systems that are compatible with $T$ are $\Hull(T)$ and the complete transfer system.

We showed in  \Cref{thm:OneCompLSP} that for an arbitrary group $G$, a connected transfer system is always lesser simply paired. In \Cref{thm:3IsNotLSP} we will show that if a $\C$-transfer system has three or more components then it is not lesser simply paired. The rest of this section is dedicated to discussing $\C$-transfer systems with two components. In this case, whether or not a transfer system is lesser simply paired depends on the shapes of the two components.

To show that a $\C$-transfer system $T$ with multiple connected components is \emph{not} lesser simply paired we need to show that there exists a $\C$-transfer system $T'$ such that $(T,T')$ is a compatible pair but $T'$ is neither $\Hull(T)$ nor the complete transfer system. We detail  our strategy for constructing such a $T'$ in the following remark.

\begin{remark}\label{rem:HowToShowNotLSP}  To search for such a $T'$, we add a new edge  to $\Hull(T)$ and let $T'$ to be the smallest transfer system such that
\begin{enumerate}
\item $T'$ contains $\Hull(T)$,
\item $T'$ contains the newly added edge, and
\item $(T,T')$ is a compatible pair.
\end{enumerate}

Note that this $T'$ exists as it is the intersection of all transfer systems satisfying these three conditions. The set of transfer systems satisfying these three conditions is nonempty as it contains $T_c$, and the intersection of two transfer systems is again a transfer system. In \Cref{lem:smallestcompatible}, we give an explicit algorithm for constructing $T'$ when the added edge has source the identity vertex $(0,0)$, which we often use in our later arguments.

Since $\Hull(T)$ does not contain the new edge, we know that $T'$ is not $\Hull(T)$. If for all possible added edges, the definition of compatibility forces $T'$ to  be the complete transfer system, then $T$ is lesser simply paired. See \Cref{ex:LSPTS}. However, if there exists an edge such that the created $T'$ is not complete, then $T$ is not lesser simply paired. See \Cref{ex:RectangleEx,ex:LShapedEx}. In practice, given that the added edge is not already in $\Hull(T)$, its source and target are in different connected components of $T$. 
\end{remark}

Following the strategy outlined above, to prove that a $\C$-transfer system $T$ with three connected components is not lesser simply paired (i.e., \Cref{thm:3IsNotLSP}) we add to $\Hull(T)$ an edge connecting  $(0,0)$ to the ``smallest" vertex not contained in $\0_T$. (Recall from \Cref{not:component} that $\0_T$ is the connected component of $(0,0)$.)  In what follows we formalize the definition of smallest vertex and develop properties of the smallest vertex of a connected component of $T$ that we use in the proof of \Cref{thm:3IsNotLSP}. We start by placing a lexicographic ordering on the vertices of the subgroup lattice of $\C$.

\begin{definition}[Lexicographically smaller, $<_L$]\label{def:lexicographic}
    Let $(a,b)$ and $(c,d)$ be vertices of the subgroup lattice of $\C$. We say $(a,b)$ is \textit{lexicographically smaller} than $(c,d)$, denoted $(a,b) <_L (c,d)$, if either $a<c$, or $a=c$ and $b<d$. 
\end{definition}

Colloquially, the above definition orders the vertices of $\C$  by moving up the columns going left to right. Hence, $(0,0)$ is the smallest vertex, $(r,s)$ is the largest vertex, and $$(i-1,s-1)<_L(i-1,s)<_L(i,0)<_L(i,1)$$ for all $i$. Further, if $T$ contains a non-trivial edge $(a,b)\to(c,d)$ then $C_{p^aq^b}$ is a subgroup of $ C_{p^cq^d}$, so $a\leq c$ and $b\leq d$. Thus, if $T$ contains the edge $(a,b) \to (c,d)$, then $(a,b)<_L (c,d)$. 

In general, $(a,b)$ being lexicographically smaller than $(c,d)$ does not imply that $a \le c$ and $b \le d$; for example, $(i-1,s)$ is lexicographically smaller than $(i,0)$. However, in the following lemma we show that if $(a,b)$ is the lexicographically smallest vertex of a connected component of a $\C$-transfer system then it is also the coordinate-wise smallest vertex of the component. We refer to this vertex as simply the smallest vertex of the component.

\begin{proposition}\label{lem:SmallestIsSmallest}
  Let $T$ be a $\C$-transfer system, and  let $(a,b)$ be the lexicographically smallest vertex of its connected component $\ij{a,b}$ in $T$. If $(x,y)$ is a vertex in $\ij{a,b}$, then $a \leq x$ and $b \leq y$.  
\end{proposition}

\begin{proof}
 Let $(a,b)$ be the lexicographically smallest vertex in $\ij{a,b}$, so for all $(x,y)$ in $\ij{a,b}$ either $a<x$ or $a=x$ and $b \leq y$. Assume for contradiction that there exists $(x,y)$ in $\ij{a,b}$ such that $a<x$ but $b>y$. (Thus, $(a,b)$ is above and to the left of $(x,y)$.)  Since $(a,b)$ and $(x,y)$ are in the same connected component, there exists a path of edges between them. Further, by the subgroup condition of \Cref{def:CtransSys}, $T$ contains neither the edge $(a,b) \to (x,y)$ nor the edge $(x,y)\to (a,b)$. Thus, the path between these two vertices must be undirected. By \Cref{lem:zigzag}, we can assume that this path has length two.
Since $(a,b)$ is the lexicographically smallest vertex, the path is of the form shown below for some vertex $(i,j)$, and thus, $(i,j)$ is above and to the right of both vertices.
\[
(a,b) \rightarrow (i,j) \leftarrow (x,y)
\]
However, by restriction of $(a,b) \to (i,j)$ with  $(x,y)$, $T$ must contain the edge  $(a,y) \to (x,y)$ as well. This means that $(a,y) \in \ij{a,b}$, but $(a,y)$ is lexicographically smaller than $(a,b)$. This is a contradiction, and thus $b \le y$. \end{proof}

\begin{definition} We define the vertex $(a,b)$ to be the \textit{smallest} vertex of a connected component of a transfer system if it is the lexicographically smallest vertex in the component. 
\end{definition}

The following result follows  from  \Cref{lem:SmallestIsSmallest}  since any connected component must have a smallest vertex, and we use it in the proof of \Cref{prop:2CompShapes} to determine all possible component shapes when a transfer system has exactly two connected components.
\begin{corollary}\label{cor:GCompIsRect}
    Let $T$ be a $\C$-transfer system. Then the connected component $\ij{r,s}$ is a rectangle. 
\end{corollary}

\begin{proof}
    By \Cref{lem:SmallestIsSmallest} the connected component of $\ij{r,s}$ has a smallest vertex, say $(a,b)$. Then every vertex  $(i,j)$ in $\ij{r,s}$  has the property $a \leq i \leq r,~ b \leq j \leq s$. Consider the (possibly undirected) path of length at most two connecting $(a,b)$ and $(r,s)$. By the subgroup condition, $(a,b)$ must be the start of any edge containing it, and $(r,s)$ must be the end. Thus, the path connecting them is of length one, and $T$ contains the edge $(a,b) \to (r,s)$. Then by restriction, if $a \leq i \leq r,~ b \leq j \leq s$, then $T$ contains $(a,b) \to (i,j)$, and hence every vertex in the rectangle is in the component of $(r,s)$. 
\end{proof}

The next lemma shows that there is an edge from the smallest vertex of a connected component of a $\C$-transfer system to every other edge in that component. This is needed both in the proof of \Cref{thm:3IsNotLSP} and in the proof of \Cref{lem:ThickNotLSP}.

\begin{lemma}\label{lem:ArromFromSmallest}
     Let $T$ be a $C_{p^rq^s}$-transfer system. Further, let $\ij{x,y}$ be the connected component of the vertex $(x,y)$, and let $(a,b)$ be the smallest vertex in  $\ij{x,y}$. Then $T$ contains the edge $(a,b)\to(x,y)$. 
    \end{lemma}

\begin{proof} 
    Consider an undirected path of length at most two between $(a,b)$ and $(x,y)$. By minimality of $(a,b)$ and the subgroup condition, any edge involving $(a,b)$ must start at $(a,b)$. Thus, if the path is of length one, we are done. If it is of length two, it is of the form
    \[(a,b) \to (u,v) \leftarrow (x,y)\]
    for some $(u,v)$, with $a\leq x \leq u$ and $b\leq y \leq v$. Restricting $(a,b)\to(u,v)$ along $(x,y)$ gives that the edge $(a,b)\to (x,y)$ is in $T$, as desired.
\end{proof}

We now prove that if a $\C$-transfer system has more than two connected components then it is not lesser simply paired.

\begin{theorem}\label{thm:3IsNotLSP}
     If a $\C$-transfer system has three or more connected components, then it is not lesser simply paired. 
\end{theorem}

\begin{proof}
Let $T$ be a transfer system with at least  three components, and  let $(i,j)$ be the lexicographically smallest vertex not in $\0_T$. (Note that every subset of the grid has a lexicographically smallest element.)
First, suppose $(i,j) \in \ij{r,s}_T$. Then since $(i,j)$ is the smallest vertex not in $\0_T$, it must be the smallest vertex in $\ij{r,s}_T$. By \Cref{lem:ArromFromSmallest}, $T$ contains the edge $(i,j) \to (r,s)$, and by restriction, contains edges $(i,j) \to (x,y)$ for all $(x,y)$ 
lexicographically larger than $(i,j)$. Since $(i,j)$ is the lexicographically smallest vertex not in $\0_T$, it follows that $T$ contains an edge from $(i,j)$ to all $(x,y)$ not in $\0_T$ and hence all $(x,y)$ not in $\0_T$ are in $\ij{r,s}_T$.  This implies $T$ has exactly two connected components, which is a contradiction.

Now suppose $(i,j)$ is not in $\ij{r,s}_T$. Using \Cref{rem:HowToShowNotLSP}, let $T'$ be the smallest  transfer system compatible with $T$ that contains $(0,0)\to (i,j)$. We show that $T'$ is not complete by showing that $T'$ does not contain an edge from $(0,0)$ to any vertex in $\ij{r,s}_T$. Since $(i,j)$ is the smallest vertex not in $\0_T$, adding in all edges required by restriction and transitivity to make $T'$ into a transfer system does not produce an edge in $T'$ from $(0,0)$ to $\ij{r,s}_T$. 

By \Cref{lem:smallestcompatible} such an edge would only come from the compatibility requirements. In order for compatibility to induce an edge between vertices in $\0$ and $\ij{r,s}_T$, following the diagrams of \Cref{rem:CpqCompDiags}, subgroups $A$ and $B$ must be in $\ij{r,s}_T$, subgroup $C$ must be in $\0$ and $T'$ must contain the edge $B\cap C \to B$. But, since $B\cap C$ must be in $\0_T$, an edge $B\cap C \to B$ in $T'$ would be an edge between $\0_T$ and $\ij{r,s}_T$. So, in order for $T'$ to contain an edge from $(0,0)$ to a vertex in $\ij{r,s}_T$, it must already contain an edge from $\0$ to $\ij{r,s}_T$, and we showed in the previous paragraph that no such edge exists.

Thus, $T'$ contains no edges $(0,0)$ to $\ij{r,s}_T$ and  is not complete. It follows that $T$ is not lesser simply paired. \end{proof}

\begin{example}\label{ex:3NotLSPEx}  \Cref{fig:3NotLSPEx} shows an example of a $C_{p^3q^2}$-transfer system $T$  with three connected components. The blue vertex $(2,0)$ is the lexicographic smallest vertex not contained in $\0_T$. The $C_{p^3q^2}$-transfer system $T'$  is the  smallest  transfer system that contains the pink edge $(0,0)\to (2,0)$ such that $(T,T')$ is a compatible pair. The green edges are the  edges in $T'$ other than $(0,0) \to (2,0)$ that are not in $T$. Note that if $T'$ contains the pink edge, then $T'$ must contain all green edges in order to satisfy the transitivity and restriction conditions of the definition of a $\C$-transfer system (\Cref{def:CtransSys}) and in order to satisfy Criterion 2  of the definition of a compatible pair of $\C$-transfer systems (\Cref{rem:CpqCompDiags}). Since $T'$ is larger than $\Hull(T)$ but is not complete, $T$ is not lesser simply paired. 
\begin{figure}[hbt!]
    \centering
    \definecolor{nicegreen}{HTML}{009B55}
\begin{tikzpicture}[>=stealth, baseline=(current bounding box.center), bend angle=20]
\fill (0,0) circle (2pt);
\fill (1,0) circle (2pt);
\fill[blue] (2,0) circle (2pt);
\fill (3,0) circle (2pt);

\fill (0,1) circle (2pt);
\fill (1,1) circle (2pt);
\fill (2,1) circle (2pt);
\fill (3,1) circle (2pt);

\fill (0,2) circle (2pt);
\fill (1,2) circle (2pt);
\fill (2,2) circle (2pt);
\fill (3,2) circle (2pt);

\fill (5,0) circle (2pt);
\fill (6,0) circle (2pt);
\fill[blue] (7,0) circle (2pt);
\fill (8,0) circle (2pt);

\fill (5,1) circle (2pt);
\fill (6,1) circle (2pt);
\fill (7,1) circle (2pt);
\fill (8,1) circle (2pt);

\fill (5,2) circle (2pt);
\fill (6,2) circle (2pt);
\fill (7,2) circle (2pt);
\fill (8,2) circle (2pt);
\node (00a) at (0,0){};
\node (10a) at (1,0){};
\node (20a) at (2,0){};
\node (30a) at (3,0){};
\node (01a) at (0,1){};
\node (11a) at (1,1){};
\node (21a) at (2,1){};
\node (31a) at (3,1){};
\node (02a) at (0,2){};
\node (12a) at (1,2){};
\node (22a) at (2,2){};
\node (32a) at (3,2){};
\node (00b) at (5,0){};
\node (10b) at (6,0){};
\node (20b) at (7,0){};
\node (30b) at (8,0){};
\node (01b) at (5,1){};
\node (11b) at (6,1){};
\node (21b) at (7,1){};
\node (31b) at (8,1){};
\node (02b) at (5,2){};
\node (12b) at (6,2){};
\node (22b) at (7,2){};
\node (32b) at (8,2){};
\draw[->] (00a) to (01a);
\draw[->, bend left] (00a) to (02a);
\draw[->] (00a) to (12a);
\draw[->] (00a) to (11a);
\draw[->] (01a) to (11a);
\draw[->] (00a) to (10a);
\draw[->] (10a) to (11a);
\draw[->, bend right] (10a) to (12a);
\draw[->] (20a) to (21a);
\draw[->, bend right] (20a) to (22a);
\draw[->] (30a) to (31a);
\draw[->, bend right] (30a) to (32a);

\draw[->] (00b) to (01b);
\draw[->, bend left] (00b) to (02b);
\draw[->] (00b) to (12b);
\draw[->] (00b) to (11b);
\draw[->] (01b) to (11b);
\draw[->] (00b) to (10b);
\draw[->] (10b) to (11b);
\draw[->, bend right] (10b) to (12b);
\draw[->] (20b) to (21b);
\draw[->, bend right] (20b) to (22b);
\draw[->] (30b) to (31b);
\draw[->, bend right] (30b) to (32b);

\draw[->, nicegreen, bend right] (00b) to (22b);
\draw[->, nicegreen] (00b) to (21b);
\draw[->, nicegreen] (01b) to (02b);
\draw[->, nicegreen] (01b) to (12b);
\draw[->, nicegreen, bend left] (01b) to (21b);
\draw[->, nicegreen] (02b) to (12b);
\draw[->, nicegreen, bend left] (02b) to (22b);
\draw[->, nicegreen] (11b) to (12b);
\draw[->, nicegreen] (21b) to (22b);
\draw[->, nicegreen] (31b) to (32b);

\draw[->, magenta, bend right] (00b) to (20b);

\draw[blue] (-0.3,-0.3)--(-0.3,2.3)--(1.3,2.3)--(1.3,-0.3)--(-0.3,-0.3);
\draw[blue] (1.7,-0.3)--(1.7,2.3)--(2.3,2.3)--(2.3,-0.3)--(1.7,-0.3);
\draw[blue] (2.7,-0.3)--(2.7,2.3)--(3.3,2.3)--(3.3,-0.3)--(2.7,-0.3);
\draw[blue] (4.7,-0.3)--(7.3,-0.3)--(7.3,2.3)--(4.7,2.3)--(4.7,-0.3);
\draw[blue] (7.7,-0.3)--(8.3,-0.3)--(8.3,2.3)--(7.7,2.3)--(7.7,-0.3);

\node at (1.5,-0.7){$T$};
\node at (6.5,-0.7){$T'$};
\end{tikzpicture}
    \caption{The transfer systems $T$ and $T'$ form a compatible pair, hence  $T$ is not lesser simply paired.}
    \label{fig:3NotLSPEx}
\end{figure}
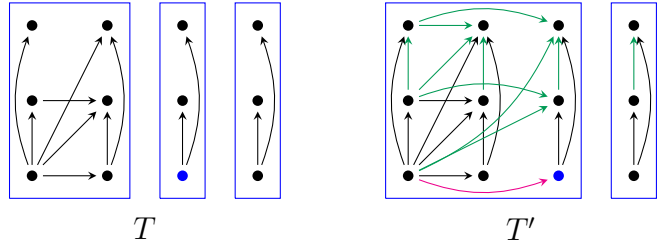
\end{example}

The remainder of this section is dedicated to the case when a $\C$-transfer system has exactly two connected components. In this situation, whether or not the transfer system is lesser simply paired depends on the shapes of the two components. In \Cref{prop:2CompShapes} we show that if a transfer system has two connected components then there are three possibilities for the shapes of the components: two horizontally stacked rectangles, two vertically stacked rectangles, or $\0$ is an L-shape and $\ij{r,s}$ is a rectangle, see \Cref{fig:2CompShapes}. We introduce some notation before stating the proposition.

\begin{definition}[$H_k$, $V_{\ell}$, and $L_{(\ell,k)}$ notation]    
\label{rem:shapes} We define three subsets of the grid of $r$ columns and $s$ rows of vertices. 
\begin{itemize}
    \item Define $V_{\ell}$ to be $\{ (i,j)  \mid \ 0\leq i \le \ell <r \text{ and } 0 \leq j \leq s \}$. In other words, $V_{\ell}$ contains the leftmost $\ell$ columns of vertices of the grid.
    \item Define $H_k$ to be $\{(i,j) \mid \ 0\leq i \leq r \text{ and } 0 \leq j \leq k<s \}$. Thus, $H_k$ contains the bottom $k$ rows of vertices of the grid.
    \item Define $L_{(\ell,k)}$ to be $ \{(i,j) \mid 0 \le i \le \ell \text{ and } 0 \le j \le s \text{ or } 0 \le i \le r \text{ and }  0 \le j \le k \}$. So, $L_{(\ell,k)}$ contains the  leftmost $\ell $ columns  and the bottom $k$ rows of vertices of the grid, forming an L-shape with the ``nook" of the L at the vertex $(\ell,k)$.
    
    \end{itemize}
\end{definition}

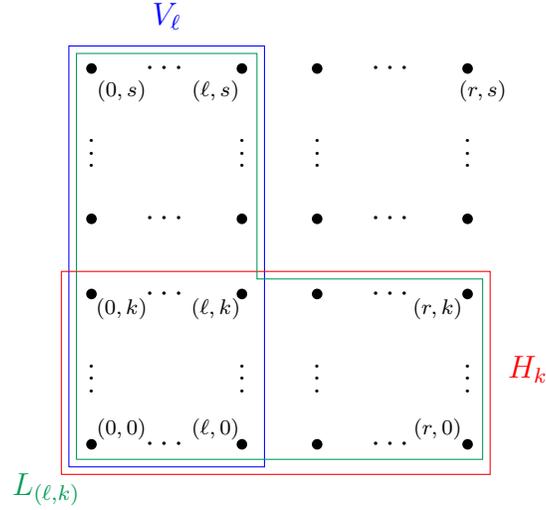
\begin{figure}
\[
\begin{tikzpicture}
\definecolor{nicegreen}{HTML}{009B55}
    \fill (0,0) circle (2pt);
    \fill (0,2) circle (2pt);
    \fill (0,3) circle (2pt);
    \fill (0,5) circle (2pt);
    \fill (2,0) circle (2pt);
    \fill (2,2) circle (2pt);
    \fill (2,3) circle (2pt);
    \fill (2,5) circle (2pt);
    \fill (3,0) circle (2pt);
    \fill (3,2) circle (2pt);
    \fill (3,3) circle (2pt);
    \fill (3,5) circle (2pt);
    \fill (5,0) circle (2pt);
    \fill (5,2) circle (2pt);
    \fill (5,3) circle (2pt);
    \fill (5,5) circle (2pt);

\node at (0,1){$\vdots$};
\node at (2,1){$\vdots$};
\node at (3,1){$\vdots$};
\node at (5,1){$\vdots$};
\node at (0,4){$\vdots$};
\node at (2,4){$\vdots$};
\node at (3,4){$\vdots$};
\node at (5,4){$\vdots$};
\node at (1,0){$\cdots$};
\node at (4,0){$\cdots$};
\node at (1,2){$\cdots$};
\node at (4,2){$\cdots$};
\node at (1,3){$\cdots$};
\node at (4,3){$\cdots$};
\node at (1,5){$\cdots$};
\node at (4,5){$\cdots$};

\draw[blue] (-0.3,-0.3)--(-0.3,5.3)--(2.3,5.3)--(2.3,-0.3)--(-0.3,-0.3);
\draw[red] (-0.4,-0.4)-- (-0.4,2.3)--(5.3,2.3)-- (5.3,-0.4)--(-0.4,-0.4);
\draw[nicegreen] (-0.2,-0.2)--(-0.2,5.2)--(2.2,5.2)--(2.2,2.2)--(5.2,2.2)--(5.2,-0.2)--(-0.2,-0.2);

\node[blue] (V) at (1,5.7){$V_\ell$};
\node[red] (H) at (5.8,1){$H_k$};
\node[nicegreen] (L) at (-0.6,-0.6){$L_{(\ell,k)}$};

\tiny
\node at (0.4,0.2){$(0,0)$};
\node at (0.4,1.8){$(0,k)$};
\node at (1.65,0.2){$(\ell,0)$};
\node at (1.65,1.8){$(\ell,k)$};
\node at (0.4,4.7){$(0,s)$};
\node at (1.65,4.7){$(\ell,s)$};
\node at (5.2,4.7){$(r,s)$};
\node at (4.6,0.2){$(r,0)$};
\node at (4.6,1.8){$(r,k)$};
\end{tikzpicture}
\]
\caption{Three important shapes of connected components, see \Cref{rem:shapes}.}
\label{fig:2CompShapes}
\end{figure}

\begin{proposition}\label{prop:2CompShapes}
     Suppose $T$ is a $\C$-transfer system with exactly  two connected components. Let $V_{\ell}$, $H_k$ and $L_{(\ell,k)}$ be as defined in \Cref{rem:shapes}. Then the connected component  $\0$ is either $V_{\ell}$ for some $\ell<r$, $H_k$ for some $k<s$, or $L_{(\ell,k)}$ for some $\ell<r$ and $k<s$. 
\end{proposition}
\begin{proof}
 Since $T$ has exactly two connected components, those components must be $\0$ and $\ij{r,s}$. By \Cref{cor:GCompIsRect}, $\ij{r,s}$ is a rectangle. Since $\0$ is the complement of $\ij{r,s}$, it follows that $\0$ must be as described in the proposition.
\end{proof}

\begin{remark}\label{rem:symmetry} We think of $V_{\ell}$ as a vertical rectangle and $H_{k}$ as a horizontal rectangle, however,  this terminology of ``vertical'' and ``horizontal'' depends on our particular choice of coordinates, i.e., the $r$ in $\C$ in the $x$-direction and the $s$ in the $y$-direction. As arbitrary $p^r$ and $q^s$ are  interchangeable, some of our statements will follow from symmetry.
\end{remark}

 We are now ready to discuss when a $\C$-transfer system with exactly two connected components is lesser simply paired. 
 We first show that if $\0$ is a horizontal rectangle with more than one row of vertices (or, equivalently, a vertical rectangle with more than one column of vertices) then $T$ is not lesser simply paired.

\begin{lemma}\label{lem:ThickNotLSP}
     Suppose $T$ is a $\C$-transfer system with exactly two connected components.  If $\0=H_k$ for some $k>0$ or $\0=V_{\ell}$ for some $\ell>0$, then $T$ is not lesser simply paired. 
\end{lemma}

\begin{proof}
Suppose $\0_T$ is $H_k$ for some $k>0$. Thus, $\0_T$ is a horizontal rectangle of height $k$ and the vertex $(0,k+1)$ is the smallest vertex in $\ij{r,s}_T$. The proof when $\0_T=V_{\ell}$ is analogous by symmetry, see \Cref{rem:symmetry}.

Let $T'$ be the smallest transfer system compatible with $T$ that contains the edge $(0,0)\to (0,k+1)$. We will use the explicit construction of $T'$ from \Cref{lem:smallestcompatible}.
By \Cref{lem:ArromFromSmallest} and transitivity, $T'$ contains an edge $(0,0)\to (x,y)$ for all $(x,y)$ in $\ij{r,s}_T$.  Using the compatibility diagrams from \Cref{rem:CpqCompDiags} with $A=(x,y)$, $B=(0,k+1)$ and $C=(a,0)$ for some $0<a<x$, it follows that $T'$ contains the edge $(a,0)\to(x,y)$ for all $(x,y)\in \ij{r,s}_T$ and all $0<a<x$.

To show that $T'$ is not complete, let $(a,b)$ be an arbitrary vertex in $\0_T$ with $b>0$ and let $(x,y)$ be an arbitrary vertex in $\ij{r,s}_T$. We will show that $T'$ does not contain the edge $(a,b)\to (x,y)$. First we note that such an edge cannot arise by restriction or transitivity  in $T'$ because there is no such edge in $T$ to begin with since the source and target lie in different components. The addition of $(0,0)\to(0,k+1)$ does not affect that.

Next, in order for such an edge to arise from compatibility, using the diagrams from \Cref{rem:CpqCompDiags}, $A$ must be $(x,y)$ and $C$ must be $(a,b)$. 
As $T$ needs to contain the edge $B \to A$, $B$ must be another vertex in $\ij{r,s}_T$.
As $C \in \0_T$ and $B \cap C \to C$ lies in $T$, $B \cap C$ must be in $\0_T$ too. Furthermore, $B \cap C \to B$ lies in $T'$ but not in $T$. 
Therefore, $B \cap C$ must be $(i,0)$ for some $0 \le i \le r$ since we showed that the edges so far obtained in $T'$ and not in $T$ from compatibility, transitivity, or restriction are of this form and $T'$ is the smallest such transfer system compatible with $T$.  
 However, if $C=(a,b)$, then $B\cap C$ cannot equal $(i,0)$ because if $B=(a_1,b_2)$, then the second coordinate of $B \cap C$ is the minimum of $b$ and $b_2$, which are both greater than 0 by assumption. Thus, compatibility does not require that $T'$ contain an edge $(a,b) \to (x,y)$. 
 
 Therefore, $T'$ is not complete, and hence $T$ is not lesser simply paired.\end{proof}

\begin{example}\label{ex:RectangleEx} \Cref{fig:2CompNotLSPEx} shows a $C_{p^3q^3}$-transfer system with two connected components such that $\0=H_1$. The transfer system $T'$ is the smallest transfer system that contains the pink edge $(0,0)\to (0,2)$ such that $(T,T')$ is a compatible pair. The green edges are the  edges in $T'$ other than $(0,0)\to (0,2)$ that are not in $T$. As in \Cref{ex:3NotLSPEx}, $T'$ must contain the green edges in order for it to be a $\C$-transfer system that is compatible with $T$. Since $T'$ is neither $\Hull(T)$ nor complete, it follows that $T$ is not lesser simply paired.
\end{example}

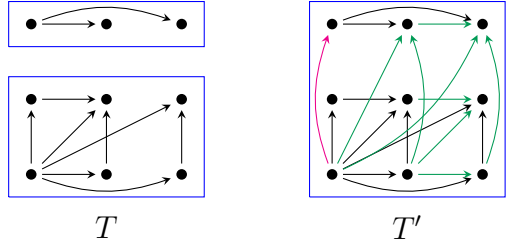
\begin{figure}[hbt!]
    \centering
    \definecolor{nicegreen}{HTML}{009B55}
\begin{tikzpicture}[>=stealth, baseline=(current bounding box.center), bend angle=20]
\fill (0,0) circle (2pt);
\fill (1,0) circle (2pt);
\fill (2,0) circle (2pt);

\fill (0,1) circle (2pt);
\fill (1,1) circle (2pt);
\fill (2,1) circle (2pt);

\fill (0,2) circle (2pt);
\fill (1,2) circle (2pt);
\fill (2,2) circle (2pt);

\fill (4,0) circle (2pt);
\fill (5,0) circle (2pt);
\fill (6,0) circle (2pt);

\fill (4,1) circle (2pt);
\fill (5,1) circle (2pt);
\fill (6,1) circle (2pt);

\fill (4,2) circle (2pt);
\fill (5,2) circle (2pt);
\fill (6,2) circle (2pt);
\node (00a) at (0,0){};
\node (10a) at (1,0){};
\node (20a) at (2,0){};
\node (01a) at (0,1){};
\node (11a) at (1,1){};
\node (21a) at (2,1){};
\node (02a) at (0,2){};
\node (12a) at (1,2){};
\node (22a) at (2,2){};
\node (00b) at (4,0){};
\node (10b) at (5,0){};
\node (20b) at (6,0){};
\node (01b) at (4,1){};
\node (11b) at (5,1){};
\node (21b) at (6,1){};
\node (02b) at (4,2){};
\node (12b) at (5,2){};
\node (22b) at (6,2){};
\draw[->] (00a) to (01a);
\draw[->, bend right] (00a) to (20a);
\draw[->] (00a) to (11a);
\draw[->] (00a) to (21a);
\draw[->] (00a) to (10a);

\draw[->] (10a) to (11a);
\draw[->] (20a) to (21a);
\draw[->] (01a) to (11a);

\draw[->] (02a) to (12a);
\draw[->,bend left] (02a) to (22a);

\draw[->] (00b) to (01b);
\draw[->, bend right] (00b) to (20b);
\draw[->] (00b) to (11b);
\draw[->] (00b) to (21b);
\draw[->] (00b) to (10b);

\draw[->] (10b) to (11b);
\draw[->] (20b) to (21b);
\draw[->] (01b) to (11b);

\draw[->] (02b) to (12b);
\draw[->,bend left] (02b) to (22b);

\draw[->, nicegreen, bend right] (00b) to (22b);
\draw[->, nicegreen] (00b) to (12b);
\draw[->, nicegreen, bend right] (10b) to (12b);
\draw[->, nicegreen] (10b) to (21b);
\draw[->, nicegreen] (10b) to (20b);
\draw[->, nicegreen, bend right] (20b) to (22b);
\draw[->, nicegreen] (11b) to (21b);
\draw[->, nicegreen] (12b) to (22b);

\draw[->, magenta, bend left] (00b) to (02b);

\draw[blue] (-0.3,-0.3)--(-0.3,1.3)--(2.3,1.3)--(2.3,-0.3)--(-0.3,-0.3);
\draw[blue] (-0.3,1.7)--(-0.3,2.3)--(2.3,2.3)--(2.3,1.7)--(-0.3,1.7);
\draw[blue]  (3.7,-0.3)--(6.3,-0.3)--(6.3,2.3)--(3.7,2.3)--(3.7,-0.3);
\node at (1,-0.7){$T$};
\node at (5,-0.7){$T'$};
\end{tikzpicture}
    \caption{The transfer system $T$ has two components with $\0=H_1$. 
    Since $T$ is compatible with $T'$, it follows that $T$ is not lesser simply paired.}
    \label{fig:2CompNotLSPEx}
\end{figure}

Next, we show that if $\0$ is L-shaped then $T$ is not lesser simply paired.

\begin{lemma}\label{lem:LNotLSP}
     Suppose $T$ is a $\C$-transfer system with exactly two connected components.  If $\0$ is $L_{(\ell,k)}$ for some $\ell<r$ and $k<s$, then $T$ is not lesser simply paired.
\end{lemma}

\begin{proof} Assume $\0_T = L_{(\ell,k)}$, so the smallest vertex in $\ij{r,s}_T$ is $(\ell+1,k+1)$. Using \Cref{lem:smallestcompatible} and \Cref{rem:HowToShowNotLSP}, let $T'$ be the the smallest transfer system compatible with $T$ that contains the edge $(0,0)\to(\ell+1,k+1)$. We will show that $T'$ is not complete. 

Since $(\ell+1,k+1)$ is the smallest edge in $\ij{r,s}_T$, by \Cref{lem:ArromFromSmallest} and transitivity, $T'$ contains all edges $(0,0)\to (i,j)$ for all $(i,j)\in \ij{r,s}_T$. By restriction, $T'$ contains all edges $(0,0)\to(i,j)$ for $i<\ell+1$ and $j<k+1$, but all such edges are already in $\Hull(T)$ (and by \Cref{prop:BigTContainsHull}, $T'$ contains $\Hull(T)$). No further edges will arise from restriction and transitivity. 

We argue that $T'$ contains no arrows other than $\Hull(T)$ and the arrows $(0,0) \to (i,j)$ for all $0 \leq i \leq r$ and $0 \leq j \leq s$.

We now consider the compatibility diagrams of \Cref{rem:CpqCompDiags}. In order to obtain further edges from compatibility, we need to choose subgroups $A$, $B$, and $C$ such that $A$ and $B$ are in $\ij{r,s}_T$, $C$ in $\ij{0,0}_T$, and $B\cap C = (0,0)$. However, since $B=(i,j)$ for $i \ge \ell+1$ and $j \ge k+1$, there is no $C$ in $\ij{0,0}_T$ such that $B\cap C = (0,0)$.  Thus, compatibility with $T$ does not require that $T'$ contain any other edges. Therefore, $T'$ is not complete, showing that $T$ is not lesser simply paired.
\end{proof}

\begin{example}\label{ex:LShapedEx}
    In \Cref{fig:LShapedNotLSPEx}, $T$ has two components with $\0=L_{(0,0)}$, and $T'$ is the smallest transfer system containing the pink edge $(0,0)\to(1,1)$ such that $(T,T')$ is a compatible pair. The green edges are the other edges in $T' \backslash T$ induced by the pink edge. It follows that $T$ is not lesser simply paired.
\end{example}

\begin{figure}[hbt!]
    \centering
    
 \definecolor{nicegreen}{HTML}{009B55}
\begin{tikzpicture}[>=stealth, baseline=(current bounding box.center), bend angle=20]
\fill (0,0) circle (2pt);
\fill (1,0) circle (2pt);
\fill (2,0) circle (2pt);

\fill (0,1) circle (2pt);
\fill (1,1) circle (2pt);
\fill (2,1) circle (2pt);

\fill (0,2) circle (2pt);
\fill (1,2) circle (2pt);
\fill (2,2) circle (2pt);

\fill (4,0) circle (2pt);
\fill (5,0) circle (2pt);
\fill (6,0) circle (2pt);

\fill (4,1) circle (2pt);
\fill (5,1) circle (2pt);
\fill (6,1) circle (2pt);

\fill (4,2) circle (2pt);
\fill (5,2) circle (2pt);
\fill (6,2) circle (2pt);
\node (00a) at (0,0){};
\node (10a) at (1,0){};
\node (20a) at (2,0){};
\node (01a) at (0,1){};
\node (11a) at (1,1){};
\node (21a) at (2,1){};
\node (02a) at (0,2){};
\node (12a) at (1,2){};
\node (22a) at (2,2){};
\node (00b) at (4,0){};
\node (10b) at (5,0){};
\node (20b) at (6,0){};
\node (01b) at (4,1){};
\node (11b) at (5,1){};
\node (21b) at (6,1){};
\node (02b) at (4,2){};
\node (12b) at (5,2){};
\node (22b) at (6,2){};
\draw[->] (00a) to (01a);
\draw[->, bend right] (00a) to (20a);
\draw[->] (00a) to (10a);
\draw[->, bend left] (00a) to (02a);
\draw[->] (10a) to (20a);
\draw[->] (01a) to (02a);

\draw[->] (11a) to (12a);
\draw[->] (11a) to (21a);
\draw[->] (11a) to (22a);

\draw[->] (00b) to (01b);
\draw[->, bend right] (00b) to (20b);
\draw[->] (00b) to (10b);
\draw[->, bend left] (00b) to (02b);
\draw[->] (10b) to (20b);
\draw[->] (01b) to (02b);

\draw[->] (11b) to (12b);
\draw[->] (11b) to (21b);
\draw[->] (11b) to (22b);

\draw[->, nicegreen, bend right] (00b) to (22b);
\draw[->, nicegreen] (00b) to (12b);
\draw[->, nicegreen] (00b) to (21b);

\draw[->, nicegreen] (12b) to (22b);
\draw[->, nicegreen] (21b) to (22b);

\draw[->, magenta] (00b) to (11b);

\draw[blue] (-0.3,-0.3)--(-0.3,2.3)--(0.3,2.3)--(0.3,0.3)--(2.3,0.3)--(2.3,-0.3)--(-0.3,-0.3);
\draw[blue] (0.7,0.7)--(0.7,2.3)--(2.3,2.3)--(2.3,0.7)--(0.7,0.7);

\draw[blue]  (3.7,-0.3)--(6.3,-0.3)--(6.3,2.3)--(3.7,2.3)--(3.7,-0.3);

\node at (1,-0.7){$T$};
\node at (5,-0.7){$T'$};
\end{tikzpicture}
    \caption{The transfer systems $T$ and $T'$ form a compatible pair, showing that $T$ is not lesser simply paired.} 
    \label{fig:LShapedNotLSPEx}
\end{figure}
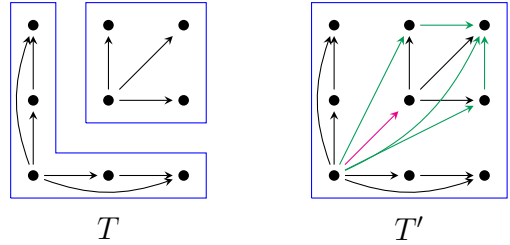

Finally, we prove that a transfer system $T$ is lesser simply paired if and only if $\0$ is either the bottom row of vertices (i.e., $H_0$) or the leftmost column of vertices (i.e., $V_0$) of $T$.

\begin{theorem}\label{thm:When2CompLSP}
     Suppose $T$ is a $\C$-transfer system with two connected components. Then $T$ is   lesser simply paired if and only if   $\0$ is $V_0$ or $H_0$.
\end{theorem}
\begin{proof} 
To prove the forward implication, assume that $\0$ is neither $V_0$ nor $H_0$. Then $T$ is not lesser simply paired by \Cref{lem:LNotLSP} or \Cref{lem:ThickNotLSP}.

For the reverse implication, assume that $\0=H_0$ and hence $\ij{r,s}=\ij{0,1}$. We will show that $T$ is lesser simply paired. The proof when $\0=V_0$ is analogous by symmetry. Following  \Cref{rem:HowToShowNotLSP}, we will show that adding \emph{any} edge to $\Hull(T)$ will generate the complete transfer system by compatibility. 

Recall that the only edges not in $\Hull(T)$ are edges between the components $\0$ and $\ij{r,s}$, so let $T'$ be the smallest compatible transfer system containing the edge $(i,0)\to (x,y)$ for an arbitrary $(i,0)$ in $H_0$ and arbitrary $(x,y)$ in $\ij{r,s}_T$.
By restriction along $(0,1)$, $T'$ contains the edge $(0,0) \to (0,1)$, and by \Cref{lem:ArromFromSmallest}, $T$ contains the edges $(0,1)\to (r,1)$ and $(0,0)\to (r,0)$. Thus, by compatibility, $T'$ contains the edge $(r,0)\to(r,1)$, see diagram below.

\[
\begin{tikzpicture}
    \node (A) at (1.6,2){$A=(r,1)$};
    \node (B) at (-1.6,2){$B=(0,1)$};
    \node (BnC) at (-1.6,0){$B\cap C=(0,0)$};
    \node (C) at (1.6,0){$C=(r,0)$};
    \draw[->] (BnC) -- (C) node[midway, above] {\tiny in $T$};
    \draw[->, blue, dashed] (C) -- (A) node[midway, right] {\tiny Need in $T'$};
    \draw[->] (B)--(A) node[midway, above] {\tiny in $T$};
    \draw[->, red] (BnC)--(B) node[midway, left] {\tiny in $T'$};
\end{tikzpicture}
\]
By restriction along $(i,1)$, $T'$ contains all edges $(i,0) \to (i,1)$. Since $T'$ contains $\Hull(T)$, \Cref{fact:HullCompsComplete} and transitivity imply that $T'$ contains all  edges, making $T'$ complete. It follows that $T$ is lesser simply paired.
\end{proof}

 \begin{example}\label{ex:LSPTS}
     In \Cref{fig:TwoCompLSPEx}, the transfer system $T$ (on the left) has two components and $\0=H_0$. We will show that if a transfer system $T'$ contains the edge $(1,0)\to(1,2)$, then in order for $(T,T')$ to be a compatible pair, $T'$ must be complete. This argument works for any transfer system that contains any edge between the two components of $T$, thus demonstrating that $T$ is lesser simply paired.

     The middle diagram  of \Cref{fig:TwoCompLSPEx} shows $\Hull(T)$ along with the edge $(1,0)\to (1,2)$ in pink. We ask what other edges must be added to the middle diagram to create a $T'$ that is compatible with $T$? First, by restriction, $T'$ must contain the green arrows $(0,0)\to (0,1)$ and $(0,0)\to (0,2)$ in the diagram on the right. Applying the compatibility diagrams of \Cref{rem:CpqCompDiags} to the  vertices labeled $A$, $B$, $C$, and $B\cap C$ shows that $T'$ must contain the dashed edge $C\to A$. Then $T'$ contains all other arrows by restriction and transitivity. Hence, $T'$ is complete.
 \end{example}

 \begin{figure}[hbt!]
     \centering
      \definecolor{nicegreen}{HTML}{009B55}
\begin{tikzpicture}[>=stealth, baseline=(current bounding box.center), bend angle=20]
\fill (0,0) circle (2pt);
\fill (1,0) circle (2pt);
\fill (2,0) circle (2pt);

\fill (0,1) circle (2pt);
\fill (1,1) circle (2pt);
\fill (2,1) circle (2pt);

\fill (0,2) circle (2pt);
\fill (1,2) circle (2pt);
\fill (2,2) circle (2pt);

\fill (4,0) circle (2pt);
\fill (5,0) circle (2pt);
\fill (6,0) circle (2pt);

\fill (4,1) circle (2pt);
\fill (5,1) circle (2pt);
\fill (6,1) circle (2pt);

\fill (4,2) circle (2pt);
\fill (5,2) circle (2pt);
\fill (6,2) circle (2pt);

\fill (8,0) circle (2pt);
\fill (9,0) circle (2pt);
\fill (10,0) circle (2pt);

\fill (8,1) circle (2pt);
\fill (9,1) circle (2pt);
\fill (10,1) circle (2pt);

\fill (8,2) circle (2pt);
\fill (9,2) circle (2pt);
\fill (10,2) circle (2pt);
\node (00a) at (0,0){};
\node (10a) at (1,0){};
\node (20a) at (2,0){};
\node (01a) at (0,1){};
\node (11a) at (1,1){};
\node (21a) at (2,1){};
\node (02a) at (0,2){};
\node (12a) at (1,2){};
\node (22a) at (2,2){};
\node (00b) at (4,0){};
\node (10b) at (5,0){};
\node (20b) at (6,0){};
\node (01b) at (4,1){};
\node (11b) at (5,1){};
\node (21b) at (6,1){};
\node (02b) at (4,2){};
\node (12b) at (5,2){};
\node (22b) at (6,2){};
\node (00c) at (8,0){};
\node (10c) at (9,0){};
\node (20c) at (10,0){};
\node (01c) at (8,1){};
\node (11c) at (9,1){};
\node (21c) at (10,1){};
\node (02c) at (8,2){};
\node (12c) at (9,2){};
\node (22c) at (10,2){};
\node at (1,-0.7){$T$};
\node at (5,-0.7){$\Hull(T)$};
\tiny
\node at (7.25,0){$B\cap C$};
\node at (7.45,1){$B$};
\node at (10.45,1){$A$};
\node at (10.45,0){$C$};

\draw[blue] (-0.3,-0.3)--(-0.3,0.3)--(2.3,0.3)--(2.3,-0.3)--(-0.3,-0.3);
\draw[blue] (-0.3,0.7)--(-0.3,2.3)--(2.3,2.3)--(2.3,0.7)--(-0.3,0.7);

\draw[blue] (3.7,-0.3)--(3.7,0.3)--(6.3,0.3)--(6.3,-0.3)--(3.7,-0.3);
\draw[blue] (3.7,0.7)--(3.7,2.3)--(6.3,2.3)--(6.3,0.7)--(3.7,0.7);

\draw[blue] (7.7,-0.3)--(7.7,0.3)--(10.3,0.3)--(10.3,-0.3)--(7.7,-0.3);
\draw[blue] (7.7,0.7)--(7.7,2.3)--(10.3,2.3)--(10.3,0.7)--(7.7,0.7);
\draw[->] (00a) to (10a);
\draw[->, bend right] (00a) to (20a);

\draw[->] (01a) to (02a);
\draw[->] (01a) to (12a);
\draw[->] (01a) to (11a);
\draw[->,bend right] (01a) to (21a);
\draw[->] (11a) to (12a);
\draw[->] (02a) to (12a);
\draw[->] (01a) to (22a);

\draw[->] (00b) to (10b);
\draw[->] (10b) to (20b);
\draw[->, bend right] (00b) to (20b);

\draw[->] (01b) to (02b);
\draw[->] (01b) to (12b);
\draw[->] (01b) to (11b);
\draw[->,bend right] (01b) to (21b);
\draw[->] (11b) to (12b);
\draw[->] (02b) to (12b);
\draw[->,bend left] (02b) to (22b);
\draw[->] (01b) to (22b);
\draw[->] (11b) to (21b);
\draw[->] (11b) to (22b);
\draw[->] (21b) to (22b);
\draw[->] (12b) to (22b);

\draw[->] (00c) to (10c);
\draw[->] (10c) to (20c);
\draw[->, bend right] (00c) to (20c);

\draw[->] (01c) to (02c);
\draw[->] (01c) to (12c);
\draw[->] (01c) to (11c);
\draw[->,bend right] (01c) to (21c);
\draw[->] (11c) to (12c);
\draw[->] (02c) to (12c);
\draw[->,bend left] (02c) to (22c);
\draw[->] (01c) to (22c);
\draw[->] (11c) to (21c);
\draw[->] (11c) to (22c);
\draw[->] (21c) to (22c);
\draw[->] (12c) to (22c);

\draw[->, nicegreen, bend left] (00c) to (02c);
\draw[->, nicegreen] (00c) to (01c);

\draw[->, magenta, bend right] (10b) to (12b);
\draw[->, magenta, bend right] (10c) to (12c);

\draw[->, thick, black,  dashed] (20c) to (21c);

\end{tikzpicture}
     \caption{The transfer system $T$ is lesser simply paired.}
     \label{fig:TwoCompLSPEx}
 \end{figure}
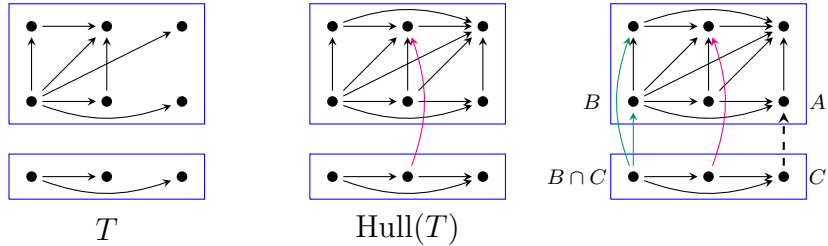
\bigskip

\begin{remark}[Lesser Simply Paired $C_{p^n}$-Transfer Systems]
    Theorems \ref{thm:OneCompLSP}, \ref{thm:3IsNotLSP}, and \ref{thm:When2CompLSP} still apply when $r=0$ or $s=0$, i.e., for $C_{p^n}$-transfer systems. Hence, a $C_{p^n}$-transfer system  $T$ is lesser simply paired if and only if either $T$ has one connected component or $T$ has two connected components and the component of the vertex $0$ consists  of a single vertex. 

    Furthermore, with this description we can enumerate the lesser simply paired transfer systems on $C_{p^n}$. A connected transfer system on $C_{p^n}$, i.e., on the vertex set $\{0,1,\cdots, n\}$ is equivalent to an arbitrary transfer system $T'$ on the vertices $\{1,\cdots, n\}$ with the edge $0 \rightarrow n$ added, and therefore, by restriction, all other edges $0 \rightarrow i$. (In the notation of \cite[Section 3.2]{NinftyOperads}, this means that a connected transfer system $T$ is of the form $T= \emptyset \odot T'$.) We know from \cite{NinftyOperads} that there are $$\Cat(n)=\frac{(2n)!}{(n+1)!n!}$$ such $T'$, where $\Cat(n)$ denotes the $n^{th}$ Catalan number. 

    Similarly, there are $\Cat(n-1)$ transfer systems on $C_{p^n}$ consisting of precisely two connected components with the component of $0$ being a single vertex as we now know how many connected transfer systems there are on the vertices $\{1,\cdots, n\}$. 
Therefore, altogether there are 
    \[
    \Cat(n) + \Cat(n-1) = \frac{5n-1}{4n-2}\cdot\Cat(n)=\frac{(5n-1)(2n-2)!}{(n-1)!(n+1)!}
    \]
    lesser simply paired transfer systems on $C_{p^n}$.

Going one step further, we   compute the proportion of $C_{p^n}$-transfer systems that are lesser simply paired (LSP);

$$\frac{\text{\# of LSP $C_{p^n}$-transfer systems}}{\text{Total \# of $C_{p^n}$-transfer systems}} = \frac{\Cat(n) + \Cat(n-1)}{\Cat(n+1)} = \frac{5n^2+9n-2}{16n^2-4}.$$

\Cref{fig:ProportTable} gives a table of the proportion of $C_{p^n}$-transfer systems that are lesser simply paired for small $n$, and evaluating $\lim_{n\to \infty} \frac{5n^2+9n-2}{16n^2-4}$ tells us that as $n$ increases, the proportion of $C_{p^n}$-transfer systems that are lesser simply paired approaches $0.3125$.

\begin{figure}[hbt!]
\begin{tabular}{llllllll}
\hline
\multicolumn{1}{|c|}{$n$}                                                                                                                          & \multicolumn{1}{c|}{2}   & \multicolumn{1}{c|}{3}   & \multicolumn{1}{c|}{4}    & \multicolumn{1}{c|}{5}    & \multicolumn{1}{c|}{6}    & \multicolumn{1}{c|}{7}    & \multicolumn{1}{c|}{8}    \\ \hline
\multicolumn{1}{|c|}{\begin{tabular}[c]{@{}c@{}}(rounded) proportion\\ of $C_{p^n}$-transfer systems\\ that are lesser simply paired\end{tabular}} & \multicolumn{1}{c|}{0.6} & \multicolumn{1}{c|}{0.5} & \multicolumn{1}{c|}{0.45} & \multicolumn{1}{c|}{0.42} & \multicolumn{1}{c|}{0.41} & \multicolumn{1}{c|}{0.39} & \multicolumn{1}{c|}{0.38} \\ \hline            
\end{tabular}
\caption{Proportion of $C_{p^n}$-transfer systems that are lesser simply paired for small $n$.}
\label{fig:ProportTable}
\end{figure}
\end{remark}

\bibliographystyle{amsalpha}

\end{document}